\documentclass[final,onefignum,onetabnum]{siamart190516}

\usepackage{colortbl}
 \definecolor{LightGray}{rgb}{0.83, 0.83, 0.83}

\usepackage{hyperref}       
\hypersetup{colorlinks=true, linkcolor=blue, breaklinks=true, urlcolor=blue}

\usepackage{amssymb,amsmath,amsfonts}
\usepackage[numbers,sort,compress]{natbib}
\usepackage{stmaryrd}
\usepackage{comment}

\usepackage{enumerate,enumitem}
\renewcommand{\theenumi}{(\roman{enumi})}

\newcommand{\metzler}[1]{\subscr{\lceil#1\rceil}{Mzr}}

\newsiamremark{remark}{Remark}
 \crefname{remark}{Remark}{Remarks}
 \newsiamremark{problem}{Problem}
 \crefname{problem}{Problem}{Problems}

 \newcommand{\real}{\mathbb{R}}
 \newcommand{\R}{\mathbb{R}}

\newcommand{\until}[1]{\{1,\dots, #1\}}
\newcommand{\fromto}[2]{\{#1,\dots, #2\}}

\newcommand{\norm}[2]{\|#1\|_{#2}}

\newcommand{\WP}[2]{\left\llbracket{#1}, {#2}\right\rrbracket}

\newcommand{\subscr}[2]{#1_{\textup{#2}}}

\newcommand{\setdef}[2]{\{#1 \; | \; #2\}}

\newcommand{\map}[3]{#1: #2 \rightarrow #3}

\newcommand{\realpart}{\operatorname{{Re}}}

\newcommand{\Iinfty}{I_{\infty}}

\newcommand{\sgn}{\operatorname{sign}}
\newcommand{\sign}{\operatorname{sign}}

\newcommand{\subject}{\text{subject to}}

\DeclareSymbolFont{bbold}{U}{bbold}{m}{n}
\DeclareSymbolFontAlphabet{\mathbbold}{bbold}
\newcommand{\vect}[1]{\mathbbold{#1}}
\newcommand{\vectorones}[1][]{\vect{1}_{#1}}
\newcommand{\vectorzeros}[1][]{\vect{0}_{#1}}

\newcommand{\mcP}{\mathcal{P}}
\newcommand{\mcY}{\mathcal{Y}}

 \usepackage[prependcaption,colorinlistoftodos]{todonotes}

\title{The Yakubovich S-Lemma Revisited:\\ Stability and Contractivity in
  Non-Euclidean Norms\thanks{Submitted to the editors DATE.\funding{This
      material was supported in part by AFOSR grant FA9550-22-1-0059.}}}

\author{Anton V. Proskurnikov\thanks{Department of Electronics and
    Telecommunications, Politecnico di Torino, Turin, Italy
    (\email{anton.p.1982@ieee.org}).}  \and Alexander
  Davydov\thanks{Department of Mechanical Engineering and the Center for
    Control, Dynamical Systems, and Computation, University of California,
    Santa Barbara, USA (\email{davydov@ucsb.edu, bullo@ucsb.edu})} \and
  Francesco Bullo\footnotemark[3] }

\headers{Non-Polynomial S-Lemma}{Proskurnikov, Davydov, and Bullo}

\begin{document}
\maketitle
\begin{abstract}
 The celebrated S-Lemma was originally proposed to ensure the existence of
 a quadratic Lyapunov function in the Lur'e problem of absolute
 stability. A quadratic Lyapunov function is, however, nothing else than a
 squared Euclidean norm on the state space (that is, a norm induced by an
 inner product). A natural question arises as to whether squared
 \emph{non-Euclidean} norms $V(x)=\|x\|^2$ may serve as Lyapunov functions
 in stability problems. This paper presents a novel non-polynomial S-Lemma
 that leads to constructive criteria for the existence of such functions
 defined by weighted $\ell_p$ norms.  Our generalized S-Lemma leads to new
 absolute stability and absolute contractivity criteria for Lur'e-type
 systems, including, for example, a new simple proof of the Aizerman and
 Kalman conjectures for positive Lur'e systems.
\end{abstract}



\section{Introduction}

\begin{figure}
  \centering
  \includegraphics[width=0.5\columnwidth]{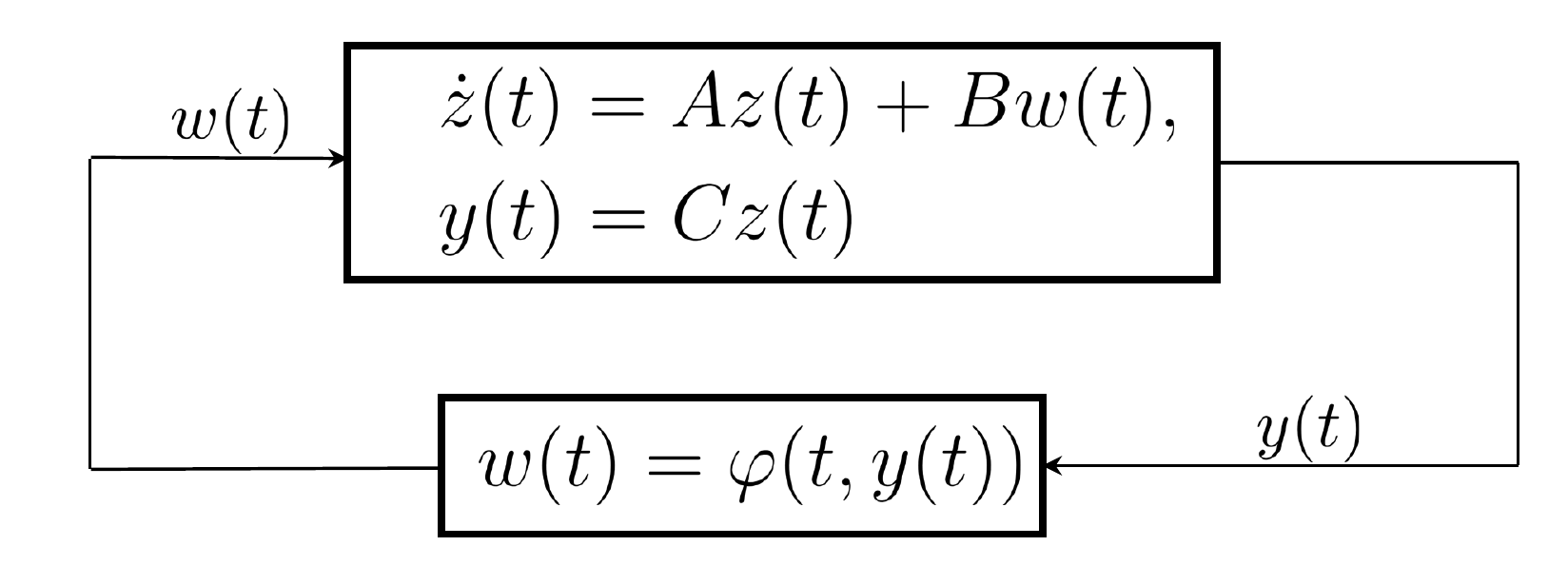}
  \caption{A Lur'e system is the feedback superposition of an LTI system and a nonlinearity}\label{fig.lurie}
\end{figure}
The history of the S-Lemma dates back to early works on stability of
nonlinear control systems with partially uncertain
dynamics~\cite{AIL:57,MAA-FRG:64}.  In many situations, such a system may
be represented in the Lur'e form, that is, as a \emph{feedback
  superposition} of two blocks as shown in Fig.~\ref{fig.lurie}. One block
is a known linear time-invariant system, whereas the other block may be
nonlinear (and is traditionally referred to as the ``nonlinearity'') and
uncertain or have no simple analytic representation as exemplified by
neural network architectures~\cite{MF-MM-GJP:20} and lookup-table
functions.  A prototypical assumption on the uncertain block is that its
input/output behavior satisfies some rough estimates. In the case of a
static nonlinearity, such an estimate often takes the form of the
\emph{sector condition}
\begin{equation}\label{eq.sector-lurie}
\alpha_1\leq\frac{w(t)}{y(t)}\leq\alpha_2\Longleftrightarrow (y(t)-\alpha^{-1}_1w(t))(y(t)-\alpha^{-1}_2w(t))\leq 0,
\end{equation}
where $-\infty\leq\alpha_1<\alpha_2\leq +\infty$. The classical Lur'e
problem~\cite{AIL:57,MRL:06} was to find conditions on the coefficients of
the known LTI block and the sector slopes $\{\alpha_1,\alpha_2\}$ that
ensure global asymptotic stability of the closed-loop system for \emph{all}
nonlinearities in the sector. Later the term \emph{absolute stability} has
been coined for such problems; the term ``absolute'' emphasizes the
applicability of the stability criteria to all unknown systems whose
``nonlinear'' parts belong to a certain class.

Historically, the first approach to absolute stability
theory~\cite{AIL:57,EAB:60} was based on quadratic Lyapunov functions and
their extensions (e.g., a quadratic form plus a definite integral of the
nonlinearity). The validation of the Lyapunov property (the Lyapunov function's derivative
along each trajectory is non-positive) leads to the following problem:
\emph{When is one quadratic inequality (the Lyapunov condition) implied by
  another quadratic inequality (the sector condition)?} More generally, when is a
quadratic inequality implied by a system of quadratic inequalities?

A sufficient condition ensuring such an implication can be obtained by a
simple trick, which was inspired by the idea of Lagrange multipliers and
termed the ``S-method''~\cite{MAA-FRG:64} or the
S-procedure~\cite{AKG-GAL-VAY:04,SVG-ALL:06}. The S-procedure reduces the
problem of a quadratic Lyapunov function existence to an linear matrix
inequality (LMI)-based condition, which can also be transcribed into a
frequency-domain stability criteria by using the Kalman-Yakubovich-Popov
(KYP) lemma. In~\cite{VAY:71}, V.\ A.\ Yakubovich established the seminal
\emph{S-Lemma} showing that, under some conditions, the S-procedure does
not introduce conservatism (i.e., the procedure is ``lossless''). This
result can also be interpreted as a strong duality theorem in a class of
non-convex optimization problems~\cite{VAY:73,VAY:92}. The S-procedure and
S-Lemma are nowadays recognized among the most important tools of modern
nonlinear control theory; their extensions and applications can be found in
the excellent surveys~\cite{SVG-ALL:06,IP-TT:07}.

When the sector inequality~\eqref{eq.sector-lurie} is replaced by the slope
inequality
\begin{equation}\label{eq.slope-lurie}
\alpha_1\leq\frac{w_1(t)-w_2(t)}{y_1(t)-y_2(t)}\leq\alpha_2,
\end{equation}
which is meant to hold for each two input-output pairs $(w_1,y_1)$ and
$(w_2,y_2)$ and all $t$, the S-Lemma often allows us to derive not only the
global asymptotic stability of the equilibrium, but in fact the stronger
property of exponential, or strong \emph{contractivity}. The strong
contractivity property of a dynamical system implies highly ordered
transient and asymptotic behavior, including (1) existence and global
exponential stability of an equilibrium for time-invariant vector fields,
(2) existence and global exponential stability of a limit cycle for
time-varying periodic vector fields, (3) input-to-state stability and
finite system gain for systems subject to state-independent disturbances as
well as robustness with respect to unmodeled dynamics, (4) modularity and
interconnection properties, and more. In other words, establishing the
contractivity property is a worthy goal, independent from stability
analysis alone.  Remarkably, one of the first contractivity criterion based
on the S-procedure was derived by V.\ A.\ Yakubovich~\cite{VAY:64} who
formulated it, however, as a criterion for entrainment (or, using the
terminology adopted in~\cite{VAY:64}, for the existence and stability of
forced periodic solutions). For comprehensive results on contractivity and
incremental stability we refer
to~\cite{WL-JJES:98,AP-AP-NVDW-HN:04,MdB-DF-GR-FS:16,HT-SJC-JJES:21,FB:22-CTDS}
and references therein.

\paragraph{\bf Problem description and motivation}

The classical S-procedure deals with quadratic inequalities, because its
primary goal was to ensure the existence of a \emph{quadratic} Lyapunov
function. A quadratic Lyapunov function $V(z)=z^{\top}Pz$, where
$P=P^{\top}$ is a constant positive definite matrix, is nothing else than the squared
\emph{Euclidean norm} on $\real^n$ (e.g., the case where $P=I_n$
corresponds to $V(z)=\|z\|_2^2$), and all Euclidean norms (norms induced by
inner products) can be represented in such a form.  Stability and
contraction analysis is usually performed by means of quadratic Lyapunov
functions whose existence boils down to feasibility of
LMIs~\cite{SB-LV:04}.

Quadratic forms, however do not exhaust the list of possible Lyapunov functions. The interest for non-Euclidean norms (e.g., $\ell_1$, $\ell_\infty$ and polyhedral norms) is more recent and motivated by classes of network
systems~\cite{FB:22-CTDS}, such as biological transcriptional systems~\citep{GR-MDB-EDS:10a}, Hopfield neural networks~\citep{YF-TGK:96,HQ-JP-ZBX:01,AD-AVP-FB:21k}, chemical reaction networks~\citep{MAAR-DA:16}, traffic networks~\citep{SC-MA:15,GC-EL-KS:15,SC:19}, vehicle platoons~\citep{JM-GR-RS:19}, and coupled oscillators~\citep{GR-MDB-EDS:13,ZA-EDS:14}.

A natural question thus arises as to whether a counterpart of the S-Lemma
exists and allows us to transcribe sector-type and other standard
constraints on the nonlinear blocks into \emph{non-quadratic} Lyapunov
inequalities arising when the Lyapunov function is chosen as
$V(z)=\|z\|^2$, where the norm is non-Euclidean. In this paper, we give an
affirmative answer and establish a \emph{non-polynomial} counterpart of the
S-Lemma, which allows us to obtain sufficient (and, in some situations,
necessary) conditions ensuring the existence of the \emph{non-quadratic}
Lyapunov function $V(z)=\|Rz\|_p^2$, where $R$ is an invertible matrix and
$p\in [1,\infty]$. The theory developed in this paper is based on the
techniques of logarithmic (log) norms and the \emph{weak
  pairings}~\cite{AD-SJ-FB:20o} associated to them. The theory of matrix
logarithmic norms~\cite{GS:06}, or ``matrix measures,'' that has been
extensively used in analysis of nonlinear circuits and systems since
the 1970s~\cite{CAD-HH:72,CAD-MV:1975}.

It should be noted that, { in the context of non-Euclidean
  norms,} the existence of a Lyapunov function $V(z)=\|Rz\|_1^2$ or
$V(z)=\|Rz\|_{\infty}^2$, where $R$ is a diagonal matrix, reduces to
checking the Hurwitz stability of some Metzler matrix~\cite{AD-AVP-FB:21k}
(whose Perron-Frobenius eigenvector determines the weight $R$). As
essentially argued for example by~\citep{AR:15}, from a computational
viewpoint, checking the Hurwitzness of a Metzler matrix is much simpler
problem (e.g., viable also for large scale problems) than solving
LMIs. Regarding algorithmic and computational aspects, \citep{WB:81} and
\citep{PVA:91} analyze and compare efficient numerical algorithms to
compute the Perron eigenvalue and eigenvector of a nonnegative irreducible
matrix.  { Beside these computational simplifications, there
  are additional practical advantages of non-Euclidean $\ell_1/\ell_\infty$ norms: (1) the $\ell_1$ norm
  (respectively, the $\ell_\infty$ norm) is well suited for systems with
  conserved quantities (respectively, systems with translation invariance),
  e.g., see the theory of weakly contracting and monotone systems in~\citep[Chapter~4]{FB:22-CTDS};
  (2) contractivity with respect to non-Euclidean norms ensures robustness
  with respect to edge removals and structural perturbations, e.g., see the
  notion of connective stability in~\cite{DDS:78}; (3) $\ell_\infty$
  contraction mappings are known to converge under the
  fully asynchronous distributed execution~\cite{DPB-JNT:91}; (4) in machine learning, analysis of the adversarial robustness of a neural net (NN) often needs to be performed in a non-Euclidean norm~\cite{BZ-DJ-DH-LW:22,YL-DH:21,SJ-AD-AVP-FB:21f}, because NNs are known to be vulnerable to small (in $\ell_{\infty}$ sense) disturbances.}

\vskip0.1cm
\paragraph{\bf Contributions} The contributions of this work are as follows:
\begin{itemize}
\item We extend the S-Lemma to a special class of non-polynomial functions
  that were introduced by Lumer~\cite{GL:61} and later used in contraction analysis~\cite{AD-SJ-FB:20o}.
\item Using the generalized S-Lemma, we derive novel criteria for absolute
  stability and contractivity of Lur'e systems. In other words, we provide
  a unifying framework for absolute stability and contractivity analysis of
  dynamical systems over normed vector spaces;
\item We demonstrate that our criteria generalize some results available in the literature, for instance, stability and
  contractivity criteria exploring symmetrization~\cite{RD-SD:18} and the
  Aizerman conjecture for positive systems~\cite{MYC:10}.
\end{itemize}
\vskip0.1cm
\paragraph{\bf Structure of the paper} Section~\ref{sec.review} introduces some mathematical concepts to be used in the subsequent sections, in particular, log norms and weak pairings associated to a norm.
Section~\ref{sec.lemma} presents our first main result (Theorem~\ref{lemma:duality}), which we call the non-polynomial S-Lemma. In Section~\ref{sec.lurie}, this result is applied to analysis of Lur'e-type systems and
establish new criteria of absolute stability and absolute contractivity in non-Euclidean norms;
new proofs of the Aizerman and Kalman conjectures for positive Lur'e systems are given.
Section~\ref{sec.concl} concludes the paper.

\section{Technical preliminaries}\label{sec.review}

We start with introducing notation. For $a\in\R$, let $a^+=\max(a,0)$.
Unless otherwise stated, vectors from $\R^n$ are considered as columns.
For two vectors $x,y\in\R^n$, the relations $\leq,\geq,<,>$ are interpreted
elementwise. The same rule applies to (equally dimensioned) matrices:
$A\leq B$ if and only if $a_{ij}\leq b_{ij}\,\forall i,j$.  The symbols
$\prec,\preceq$ apply to \emph{symmetric} matrices: we write $A\prec B$
(respectively, $A\preceq B$) if $B-A$ is positive definite (respectively,
semidefinite).

Given a matrix $A=(a_{ij})$, denote $|A|=(|a_{ij}|)$.
For two identically sized matrices $A=(a_{ij}), B=(b_{ij})$, $A\circ B=(a_{ij}b_{ij})$ denotes their Hadamard (entry-wise) product.

Recall that a matrix is \emph{Metzler} if all its off-diagonal terms are non-negative. Given a matrix $A\in\real^{n\times{n}}$, its \emph{Metzler majorant} $\metzler{A}\in\real^{n\times{n}}$ is defined by
\[
(\metzler{A})_{ij}: =
\begin{cases} a_{ii},\quad & \text{if }i=j \\ |a_{ij}|, \quad &\text{if } i\neq j.
\end{cases}
\]
Obviously, a matrix is Metzler if and only if it coincides with its Metzler majorant.

\subsection{Log norms and weak pairings in normed spaces}
Let $\|\cdot\|$ be a norm on $\R^n$; the same symbol will be used to denote the induced operator norm on
$\real^{n\times{n}}$.  The \emph{log norm} of $A \in \R^{n \times n}$
with respect to $\norm{\cdot}{}$ is
\begin{equation}
  \mu(A):= \lim_{h \to 0^+} \frac{\|I_n + hA\| - 1}{h}.
\end{equation}
The \emph{conic log norm}~\cite{SJ-AD-FB:20r} of $A\in \real^{n\times n}$ is
  \begin{align*}
    \mu^{+}(A) := \lim_{h\to 0^+}\sup_{\substack{x\ge \vectorzeros[n]\\x\ne \vectorzeros[n]}}\frac{\|(I_n + h A)x\|/\|x\| -1}{h}.
  \end{align*}
We refer to~\cite{CAD-HH:72,AD-SJ-FB:20o} and~\cite{SJ-AD-FB:20r} for, the theory of
log norms and conic log norms.

From~\cite{AD-SJ-FB:20o}, a \emph{weak pairing} on $\R^n$ is a map
$\WP{\cdot}{\cdot}: \R^n \times \R^n \to \R$ satisfying
\begin{enumerate}
\item\label{WSIP1}(Subadditivity and continuity of first argument) $\WP{x_1+x_2}{y} \leq \WP{x_1}{y} + \WP{x_2}{y}$, for all $x_1,x_2,y \in \R^n$ and $\WP{\cdot}{\cdot}$ is continuous in its first argument,
\item\label{WSIP3}(Weak homogeneity) $\WP{\alpha x}{y} = \WP{x}{\alpha y} = \alpha\WP{x}{y}$ and $\WP{-x}{-y} = \WP{x}{y}$, for all $x,y \in \R^n, \alpha \geq 0$,
\item\label{WSIP4}(Positive definiteness) $\WP{x}{x} > 0$, for all $x \neq
  \vectorzeros[n],$
\item\label{WSIP5}(Cauchy-Schwarz inequality) $|\WP{x}{y}| \leq \WP{x}{x}^{1/2}\WP{y}{y}^{1/2}$, for all $x, y \in \R^n.$
\end{enumerate}
A weak pairing is \emph{compatible} with a norm $\norm{\cdot}{}$ if
$\WP{x}{x}=\|x\|^2$ for all $x$. Table~\ref{table:equivalences}  contains
weak pairings compatible with every $\ell_p$ norm, $p\in[1,\infty]$~\cite{AD-SJ-FB:20o,SJ-AD-FB:20r}. This
list includes the \emph{sign pairing} for $\ell_1$ norm and the \emph{max
  pairing} for the $\ell_\infty$ norm. Only unweighted $\ell_p$ norms are
included since $\mu_{p,R}(A) = \mu_{p}(RAR^{-1})$ for any $p \in
[1,\infty]$.

\begin{table}\centering
  \normalsize
    \resizebox{1\textwidth}{!}
            {\begin{tabular}{%
	p{0.28\linewidth}%
	p{0.28\linewidth}%
	p{0.4\linewidth}%
      }
      Norm & Weak pairing
      & Log norms and Lumer's equality
      \\
      \hline
      \rowcolor{LightGray}    &&  \\[-2ex]
      \rowcolor{LightGray}
      $ \begin{aligned}
	\norm{x}{2} = \sqrt{x^\top  x}
      \end{aligned}$
      &
      $\begin{aligned}
	\WP{x}{y}_{2} &= x^\top y
      \end{aligned}$
      &
      $\begin{aligned}
	\mu_{2}(A)  &= \tfrac{1}{2}\lambda_{\max}(A + A^\top)
	\\ &= \max_{\|x\|_{2}=1} x^\top Ax
      \end{aligned}
      $
      \\[2ex]
      &&  \\[-2ex]
      $ \begin{aligned}
	&\norm{x}{p} = \Big(\sum_{i} |x_i|^p\Big)^{1/p}\!\!\!,\\
&1<p<\infty
      \end{aligned}$
      &
      $\begin{aligned}
	&\WP{x}{y}_{p} =\\
&\phantom{123}\|y\|_{p}^{2-p}(y \circ |y|^{p-2})^\top x
      \end{aligned}$
      &
      $\begin{aligned}
	\mu_{p}(A) &= \max_{\|x\|_{p} = 1} (x \circ |x|^{p-2})^\top Ax
      \end{aligned}
      $
      \\
      \rowcolor{LightGray}    &&  \\[-2ex]
      \rowcolor{LightGray}
      $ \begin{aligned}
	\norm{x}{1} &= \sum_i |x_i|
      \end{aligned}$
      &
      $\begin{aligned}
	\WP{x}{y}_{1} &= \|y\|_{1}\sign(y)^\top x
      \end{aligned}$
      & $\begin{aligned}
	\mu_{1}(A) &= \max_{j \in \until{n}} \Big(a_{jj} + \sum_{i \neq j} |a_{ij}|\Big) \\
                        &= \sup_{\|x\|_{1} = 1} \sign(x)^\top Ax\\
    \mu_1^+(A)&=\max_{j \in \until{n}} \Big(a_{jj} + \sum_{i \neq j} a_{ij}^+\Big) \\
                        &= \sup_{\|x\|_{1} = 1,x\geq \vectorzeros[n]} \sign(x)^\top Ax\\
      \end{aligned}$
      \\[2ex]
      &&  \\[-2ex]
      $ \begin{aligned}
	\norm{x}{\infty} &= \max_i |x_i|
      \end{aligned}$
      &
      $\begin{aligned}
	\WP{x}{y}_{\infty} &= \max_{i \in I_{\infty}(y)} x_iy_i
      \end{aligned}$
      & $\begin{aligned}
	\mu_{\infty}(A) &= \max_{i \in \until{n}} \Big(a_{ii} + \sum_{j \neq i} |a_{ij}|\Big) \\
                                &=\sup_{\|x\|_{\infty} = 1} \max_{i \in I_{\infty}(x)} (Ax)_ix_i\\
    \mu_{\infty}^+(A) &= \max_{i \in \until{n}} \Big(a_{ii} + \sum_{j \neq i} a_{ij}^+\Big) \\
                                &=\sup_{\|x\|_{\infty} = 1,x\geq \vectorzeros[n]} \max_{i \in I_{\infty}(x)} (Ax)_ix_i
      \end{aligned}$
      \\[2ex] \hline
      && \\[-1ex]
  \end{tabular}}
  \caption{Table of norms, weak pairings, and log norms for 
    $\ell_2$, $\ell_p$ for $p\in{(1,\infty)}$, $\ell_1$, and $\ell_\infty$
    norms.  We adopt the shorthand $\Iinfty(x) =
    \setdef{i\in\until{n}}{|x_i|=\norm{x}{\infty}}$.} \label{table:equivalences}
\end{table}

The pairings in Table~\ref{table:equivalences} additionally satisfy Lumer's
equalities and the curve norm derivative formulas~\cite{AD-SJ-FB:20o,SJ-AD-FB:20r}.  For
all $A \in \real^{n\times n}$, \emph{Lumer's equalities} state that
\begin{gather}
  \mu(A) = \sup_{\|x\| = 1} \WP{Ax}{x} = \sup_{x \neq \vectorzeros[n]} \frac{\WP{Ax}{x}}{\|x\|^2},\label{eq:Lumer}\\
  \mu^+(A) = \sup_{\|x\| = 1,\,x\geq \vectorzeros[n]} \WP{Ax}{x} = \sup_{\substack{x\ge \vectorzeros[n]\\x\ne \vectorzeros[n]}} \frac{\WP{Ax}{x}}{\|x\|^2}\label{eq:Lumer+}
\end{gather}
For Metzler matrices and $\ell_1$ norm, the log-norm and the conic-log norm are, obviously, coincident, furthermore, the closed simplex $\{x:\|x\| = 1,\,x\geq \vectorzeros[n]\}$
in~\eqref{eq:Lumer+} can be replaced by its (relative) interior as shown by the following proposition.

\begin{proposition} \label{prop.M-l1}
For every $n\times n$ Metzler matrix $M$, one has
  \begin{equation}\label{eq.conic-metzler}
    \mu_1(M)=\mu_{1}^+(M)=\sup_{\substack{\norm{x}{1}=1\\x>\vectorzeros[]}}\WP{Mx}{x}_1=\sup_{x>\vectorzeros[n]} \frac{\WP{Mx}{x}}{\|x\|^2}.
  \end{equation}
  \end{proposition}

\emph{The curve norm derivative formula} states that for every differentiable $x: {(a,b)} \to \R^n$ and for almost every $t \in
{(a,b)},$ that right upper Dini derivative of $\|x(t)\|^2$ at $t$ is
\begin{equation}
    \label{eq:cndf}
  D^+ \|x(t)\|^2:  = \limsup_{h\to 0+}\frac{\|x(t+h)\|^2-\|x(t)\|^2}{h}=2\WP{\dot{x}(t)}{x(t)}.
\end{equation}
The curve norm formula is typically used to differentiate a Lyapunov function $V(x)=\|x\|^2$ along the system's trajectories, proving stability or contraction~\cite{AD-SJ-FB:20o}. Unless otherwise stated, the curve norm derivative formula is always supposed to hold. As shown in~\cite{AD-SJ-FB:20o}, all WPs from Table~\ref{table:equivalences} enjoy this useful property. For these WPs, a useful equality holds~\cite[Appendix~C]{AD-SJ-FB:20o}: for all $x,y\in\real^n$, and $c\in\real$
\begin{equation}\label{eq:useful-eq}
  \WP{x + cy}{y} = \WP{x}{y} + c\norm{y}{}^2.
\end{equation}

\begin{remark}
Notice that~\eqref{eq:useful-eq} \emph{does not} guarantee that the WP is linear in its first argument, that is,
\begin{equation}\label{eq:linear}
\WP{ax_1+bx_2}{y}=a\WP{x_1}{y}+b\WP{x_2}{y}
\end{equation}
for all vectors $x_1,x_2,y$ and scalars $a,b\in\mathbb{R}$. For instance, $\WP{\cdot}{\cdot}_{\infty}$, 
as can be seen from Table~I, satisfies~\eqref{eq:useful-eq} yet fails to be linear in its first argument.
\end{remark}

\subsection{Non-polynomial 2-forms associated to a weak pairing}

A standard quadratic form on $\real^n$ admits two equivalent representations. On one hand, it can be considered as a
a homogeneous polynomial $q(x_1,\ldots,x_n)=\sum_{i,j}q_{ij}x_ix_j$ of degree $2$ also termed as a \emph{polynomial 2-form}. The term \emph{homogeneous} means that all (non-zero) terms have same degree $d$ (in our situation $d=2$), or, equivalently, for each scalar $\lambda\in\real$ one has $q(\lambda x_1,\ldots,\lambda x_n)=\lambda^dq(x_1,\ldots,x_n)$. On the other hand, the quadratic form can be considered as a function $q(x)=x^{\top}Qx=\WP{Qx}{x}_2$, where $Q$ is a matrix.

Given a weak pairing $\WP{\cdot}{\cdot}$ and a matrix $P$, define the
\emph{non-polynomial 2-form}\footnote{Such functions were first introduced by Lumer~\cite{GL:61} in the special case where $\WP{\cdot}{\cdot}$ is a semi-inner product on a normed space (possibly, infinite-dimensional).}
\begin{equation*}
  p(x)=\WP{P x}{x},\quad x\in\R^n.
\end{equation*}
For the standard $\ell_1$, $\ell_2$, and $\ell_\infty$ norms, we have:
\begin{align*}
  p_1(x)&=\WP{P x}{x}_1= \|x\|_{1}\sign(x)^\top Px,\\
  p_2(x)&=\WP{P x}{x}_2= x^\top Px ,\\
  p_\infty(x)&=\WP{P x}{x}_\infty=  \max_{i \in I_{\infty}(x)} x_i(Px)_i ,\quad \Iinfty(x) =
\setdef{i\in\until{n}}{|x_i|=\norm{x}{\infty}}.
\end{align*}
For brevity, we omit the term ``non-polynomial'' and call $p(x)$ simply 2-form.

\section{Non-polynomial S-Lemma for general normed spaces}\label{sec.lemma}

Consider a WP $\WP{\cdot}{\cdot}$ on $\real^n$ compatible with a norm
$\|\cdot\|$ on vectors and log norm $\mu(\cdot)$ on matrices. We assume the
weak pairing satisfies Lumer's equality and the curve norm derivative
formula. Consider a family of $s+1$ matrices
$P_0,\ldots,P_s\in\real^{n\times n}$ and functions
\begin{equation}
  p_i(x)=\WP{P_ix}{x},\quad i\in\fromto{0}{s}.
\end{equation}
Given a constant $\rho\in\real^s$, we define the primal optimization
problem
\begin{equation}\label{prob:primal}
  \begin{split}
    \sup_{x\in\real^n}& \quad p_0(x)\\
    \subject &\quad \|x\|=1,\enspace p_1(x)\leq\rho_1,\enspace\ldots \enspace, \enspace p_s(x)\leq \rho_s,
  \end{split}
\end{equation}
and the dual optimization problem:
\begin{equation}\label{prob:dual}
  \begin{split}
    \inf_{\tau\in\real^s}& \quad \mu\Big(P_0-\sum_{j=1}^s\tau_jP_j\Big)+\tau^{\top}\rho\\
    \subject &\quad \tau\geq \vectorzeros[s].
  \end{split}
\end{equation}

We note that the primal problem~\eqref{prob:primal} is non-convex and the constraints, in general, are feasible only for sufficiently large $\rho_i$.
By definition, let $-\infty$ be the value of the optimization problem~\eqref{prob:primal} if the constraints are infeasible.
On the other hand,~\eqref{prob:dual} is a convex program with feasible constraints no matter how the norm is chosen.
In the case of $\ell_2$ norm~\eqref{prob:dual} is a standard semidefinite program, whereas in $\ell_1/\ell_{\infty}$ it turns out to be a linear program,
which allows us to solve it efficiently.

\subsection{Non-polynomial S-Lemma: a weak duality result}

The standard relation of \emph{weak duality} (duality without zero-gap
guarantee)~\cite{SB-LV:04} entails that the infimum in~\eqref{prob:primal}
is not less than the supremum in~\eqref{prob:dual}, provided that the WP is linear in the first argument 
{in the sense of~\eqref{eq:linear}.}
In reality, the linearity requirement can be discarded  as shown by
the following lemma, which is a non-polynomial counterpart of the Yakubovich S-Lemma.


\begin{theorem}[Non-polynomial S-Lemma: Weak duality for Non-Euclidean norms]
  \label{lemma:duality}
  Let $\norm{\cdot}{}$ be a norm on $\real^n$ with log norm $\mu(\cdot)$
  and compatible weak pairing $\WP{\cdot}{\cdot}$ satisfying Lumer's
  equality. Given $P_0,\ldots,P_s\in\real^{n\times n}$ and $\rho\in\real^s$
  \begin{enumerate}
  \item\label{l:d:1} if $\WP{\cdot}{\cdot}$ is linear~\eqref{eq:linear} in its first
    argument, then the optimization problem~\eqref{prob:dual} is the
    Lagrangian dual problem to the optimization
    problem~\eqref{prob:primal};
  \item\label{l:d:2} for an arbitrary $\WP{\cdot}{\cdot}$, let $\alpha$ and $\beta$ denote, respectively, the
    supremum in~\eqref{prob:primal} and the infimum
    in~\eqref{prob:dual}. Then $\alpha\leq\beta$.
    \item \label{l:d:3} In particular, the following statement is valid:
    \begin{equation}\label{eq.S-proc-convenient}
    p_0(x)\leq\beta\|x\|^2\quad\forall x\in\R^n: p_1(x)\leq\rho_1\|x\|^2,\ldots,p_s(x)\leq\rho_s\|x\|^2.
    \end{equation}
  \end{enumerate}
\end{theorem}
\begin{proof}
  Adopt the short-hand $P(\tau)=P_0-\sum_{i=1}^s\tau_iP_i$.  Following the
  notation from~\cite{VAY:92}, define the Lagrangian function
  $\map{L}{\real^n\times\real^s}{\real}$ by
  \begin{gather*}
    L(x,\tau)=-p_0(x)+\sum_{j=1}^s\tau_jp_j(x)-\sum_{j=1}^s\tau_j\rho_j=-\tau^{\top}\rho-\WP{P(\tau)x}{x}.
  \end{gather*}
  Applying Lumer's equality~\eqref{eq:Lumer}, the Lagrange dual
  function~\cite{SB-LV:04} becomes
  \begin{equation*}
    g(\tau)=\inf_{x\in\real^n, \norm{x}{}=1}L(x,\tau)=
    -\tau^{\top}\rho-\sup_{x\in\real^n, \norm{x}{}=1}\WP{P(\tau)x}{x}=-\tau^{\top}\rho-\mu(P(\tau)).
  \end{equation*}
  The Lagrange dual problem to~\eqref{prob:primal} is written as
  \begin{equation}\label{eq.dual-}
      \sup_{\tau\in\real^s}\;\;g(\tau)\;\;
      \subject \;\;\tau_1,\ldots,\tau_s\geq 0,
  \end{equation}
  or, equivalently, as problem~\eqref{prob:dual}.  This concludes the proof
  of statement~\ref{l:d:1}.

   For each vector $x\in\real^n$ satisfying constraints~\eqref{prob:primal} and every $\tau\geq\vectorzeros[s]$ one has
  \begin{align}
    \WP{P_0x}{x}&=\WP{P_0x-\sum_{i=1}^s\tau_iP_ix+\sum_{i=1}^s\tau_iP_ix}{x} \nonumber \\
    &\leq
    \WP{P_0x-\sum_{i=1}^s\tau_iP_ix}{x} +    \sum_{i=1}^s\tau_i  \WP{P_ix}{x}
    && \text{(by subadditivity of $\WP{\cdot}{\cdot}$)}
    \nonumber \\
    &\leq \WP{\left(P_0-\sum_{i=1}^s\tau_iP_i\right)x}{x} + \sum_{i=1}^s\tau_i\WP{P_ix}{x}
    && \text{(by weak homogeneity of $\WP{\cdot}{\cdot}$)}
    \nonumber \\
    & \leq \mu(P(\tau))+\sum_{i=1}^s\tau_i\rho_i. \label{eq.duality}
    && \text{(by Lumer's equality for $\WP{\cdot}{\cdot}$)}
  \end{align}
  Taking the supremum of $\WP{P_0x}{x}$ over all feasible $x$ and the infimum of the right-hand side over all $\tau\geq\vectorzeros[s]$, one proves that $\alpha\leq\beta$. This concludes
  the proof of statement~\ref{l:d:2}.

  Statement~\ref{l:d:3} is straightforward from~\ref{l:d:2} by noticing that $p_0(x)\leq\alpha\|x\|^2$ for all $x\in\R^n$ satisfying the inequalities $p_i(x)\leq\rho_i\|x\|^2$:
  this statement is obvious for $x=0$, otherwise, the normalized vector $\tilde x=x/\|x\|$ obeys the constraints in~\eqref{prob:primal} and thus $p_0(x)=\|x\|^2p_0(\tilde x)\leq\alpha\|x\|^2$ due to the definition of $\alpha$.
\end{proof}

\begin{remark}[Conic constraints and more]\label{rem.positive}
  Theorem~\ref{lemma:duality} can be extended with a
  trivial modification to the optimization problems with an additional
  conic constraint $x\geq \vectorzeros[s]$. In this case, the primal
  problem is
  \begin{equation}\label{prob:primal+} 
    \begin{split}
      \sup_{x\in\real^n}& \quad p_0(x) \\
      \subject &\quad \|x\|=1, x\geq \vectorzeros[n] , \enspace p_1(x)\leq\rho_1,\enspace\cdots \enspace, \enspace p_s(x)\leq \rho_s,
    \end{split}
  \end{equation}
  and the usual log norm in the dual problem~\eqref{prob:dual} is replaced
  by the conic log norm:
  \begin{equation}\label{prob:dual+} 
  \begin{split}
    \inf_{\tau\in\real^s}& \quad \mu^+\Big(P_0-\sum_{j=1}^s\tau_jP_j\Big)+\tau^{\top}\rho\\
    \subject &\quad \tau\geq \vectorzeros[s].
  \end{split}
\end{equation}
Denoting the latter infimum by $\beta^+$, the inequality~\eqref{eq.S-proc-convenient} should be rewritten as
\begin{equation}\label{eq.S-proc-convenient+}
    p_0(x)\leq\beta^+\|x\|^2\quad\forall x\in\R^n_{\geq 0}: p_1(x)\leq\rho_1\|x\|^2,\ldots,p_s(x)\leq\rho_s\|x\|^2.
    \end{equation}

\end{remark}

\begin{remark}[Equivalent primal constraints]\label{rem.0}
  From~\eqref{eq:useful-eq}, we know $\WP{y-\rho
    x}{x}=\WP{y}{x}-\rho\norm{x}{}^2$ for each $\rho\in\real$. This
  equality allows us to simplify the constraints in~\eqref{prob:primal} by
  replacing $P_j\mapsto P_j-\rho_j I_j$ and $\rho_j\mapsto 0$. However, sometimes (see below the case of positive systems) it is more
  convenient to consider the general situation.
\end{remark}

\begin{remark}[No constraints implies no gap]\label{rem:no-constraints=no-gap}
  If the primal problem is unconstrained, that is, $s=0$, then
  $\alpha=\beta=\mu(P_0)$ because of Lumer's equality~\eqref{eq:Lumer}. We study two other
  special cases with zero-gap in Subsections~\ref{subsec.l2} and~\ref{subsec.l2}.
  In general, the \emph{duality gap} the optimal values exists: $\alpha<\beta$ (see the examples below).
\end{remark}

\subsection{The classic S-Lemma: the case of Euclidean norm $\ell_2$}\label{subsec.l2}

The first case where there is no gap is where the primal problem has only one constraint ($s=1$) and the norm is Euclidean (i.e., induced by the inner product).
The following result is equivalent to~\cite[Theorem~5.17]{IP-TT:07} and to the Yakubovich S-Lemma~\cite{VAY:73}. Without loss of generality
(Remark~\ref{rem.0}), let $\rho_1=0$.

\begin{lemma}[The Yakubovich S-Lemma]\label{lem.yakub}
  Suppose that $\WP{\cdot}{\cdot}$ is an inner product, $s=1$, and the
  primal constraint reads $\WP{P_1x}{x}\leq 0$. Suppose also that for some $x$ the latter inequality is strict:
  $\WP{P_1x}{x}<0$. Then
  \begin{enumerate}
  \item the maximum and minimum points of  primal and dual problems exist, and there is no duality gap: $\alpha=\beta$ or, equivalently
    \begin{equation}
      \max\setdef{\WP{P_0x}{x}}{\|x\|=1,\,\WP{P_1x}{x}\leq
        0}=\min\setdef{\mu( P_0 - \tau P_1)}{\tau\geq 0},
    \end{equation}
  \item if $x_*$ is a maximizer (generally, non-unique)
    in~\eqref{prob:primal} and $\tau_*$ is a minimizer
    in~\eqref{prob:dual}, then the complementarity condition holds
    \[
    \tau_*\WP{P_1x_*}{x_*}=0.
    \]
  \end{enumerate}
\end{lemma}
\begin{proof}
  The proof retraces the proofs of Theorem~2 in~\cite{VAY:73}, see
  also~\cite[Appendix~A]{VAY:92}, and is based on the S-Lemma for two
  quadratic forms.  Since the inequality $p_1(x)<0$ has at least one
  solution, it also holds for some $x$ from the sphere $\norm{x}{2}\leq
  1$. Hence, problem~\eqref{prob:primal} (with $s=1$) is feasible and
  $\alpha>-\infty$.  The set of admissible vectors $x$
  in~\eqref{prob:primal} is compact, and hence the supremum can be replaced
  by the maximum.\footnote{Notice that the scalar product, unlike general
    WPs, is continuous in both arguments.} Since the functions $p_0,p_1$
  are homogeneous,
  \[
  p_1(x)\leq 0 \quad\implies\quad p_0(x)-\alpha\norm{x}{2}^2<0\quad\forall x\in\real^n.
  \]
  Recalling that the set $\setdef{x}{p_1(x)<0}$ is
  non-empty, the S-Lemma \cite[Theorem~1]{VAY:73}
  and~\cite[Theorem~2.2]{IP-TT:07} ensures that a number $\tau_*\geq 0$
  exists such that
  \[
  p_0(x)-\alpha\norm{x}{2}^2-\tau_* p_1(x)\leq 0,
  \]
  whence we obtain that
  \[
  \mu(P(\tau_*))=\sup_{\norm{x}{2}=1}(p_0(x)-\tau_* p_1(x))\leq\alpha.
  \]
  On the other hand, $\mu(P(\tau))\geq\alpha$ for any $\tau\geq 0$ due to
  the non-polynomial S-Lemma (Theorem~\ref{lemma:duality}). Hence, $\tau_*$ is a minimizer
  in~\eqref{prob:dual} (where $\rho_1=0$). The last statement can be easily
  derived from~\eqref{eq.duality}, by noticing that
  $p_0(x_*)=\mu(P(\tau_*))=\alpha$.
\end{proof}

\begin{remark}[Inhomogeneous quadratic functions]
Lemma~\ref{lem.yakub} remains valid~\cite{IP-TT:07} for more general inhomogeneous quadratic functions $\WP{P_ix+q_i}{x}+r_i$, 
where $P_i$ are constant matrices, $q_i$ are vectors and $r_i\in\mathbb{R}$ ($i=0,1$). Such quadratic constraints arise in 
many problems of robust control and optimal control~\cite{GB-JZ-DdH-JL:20}. Generalization of Theorem~\ref{lemma:duality} to 
such functions in the case where $\WP{\cdot}{\cdot}$ is a general weak pairing remains an open problem.
\end{remark}

\begin{remark}[Other zero-gap cases]
  In the Euclidean norm case some other situations (with $s\geq 2$) are
  known where the duality gap vanishes: this holds (under minor assumptions) when the
  image of the quadratic mapping $(p_0(x),\ldots,p_s(x))$ is convex. This
  counter-intuitive convexity always takes place when $s=1$ (two quadratic
  forms)~\cite{LLD:41}, being also a feature of some specially structured
  quadratic functions when $s>1$. For a detailed survey of recent
  achievements in the area, we refer the reader to~\cite{SVG-ALL:06,IP-TT:07}.
  As noted in~\cite{VAY:71,VAY:92}, Lemma~\ref{lem.yakub} can be
  extended to every such situation.
\end{remark}

\subsection{The S-Lemma for Metzler matrices in $\ell_1$ norm}\label{subsec.l1}

In this section we consider another situation where the duality gap
vanishes. Consider the situation of $\ell_1$-norm with the sign weak
pairing $\WP{x}{y}_1=\|y\|_1\sign(y)^\top x$. Notice that this function is
discontinuous in $y$.  Along with the ``non-negative'' optimization
problem~\eqref{prob:primal+} consider the problem with stricter constraints
\begin{equation}\label{prob:primal++}
      \begin{split}
      \sup_{x\in\real^n}& \quad \WP{P_0x}{x}_1 \\
      \subject &\quad \norm{x}{1}=1, x>\vectorzeros[n], \enspace \WP{P_1x}{x}_1<\rho_1,\ldots,\WP{P_sx}{x}_1<\rho_s.
  \end{split}
\end{equation}

  The optimization problems~\eqref{prob:primal} and~\eqref{prob:primal+} are
  non-convex, whereas the objective function and the constraints
  in~\eqref{prob:primal++} are linear, because
  $\WP{P_ix}{x}_1=\vectorones^{\top}P_ix$ for all $x\in\real_{>0}^n$ with $\|x\|_1 = 1$.  At
  the same time,~\eqref{prob:primal++} is not a standard LP problem due to
  the presence of strict inequalities.

\begin{lemma}[S-Lemma for Metzler matrices in $\ell_1$ norm]\label{lem.nogap-positive}
  Assume that matrices $P_0$ and   $(-P_1),(-P_2),\ldots,(-P_s)$ are Metzler and that the constraints
  in~\eqref{prob:primal++} are feasible. Then the following values
  coincide:
  \begin{enumerate}
  \item $\alpha :=$  supremum in~\eqref{prob:primal}  (where $\|\cdot\|=\|\cdot\|_1$);
  \item $\alpha^+ :=$  supremum in~\eqref{prob:primal+} (where $\|\cdot\|=\|\cdot\|_1$);
  \item $\alpha^{++} :=$   supremum in~\eqref{prob:primal++} (where $\|\cdot\|=\|\cdot\|_1$);
  \item $\beta^+ :=$   infimum in~\eqref{prob:dual+} (where $\mu^+=\mu_{1}^+$);
  \item $\beta:=$ infimum in~\eqref{prob:dual}  (where $\mu=\mu_{1}$).
  \end{enumerate}
Furthermore, the minimum value in problem~\eqref{prob:dual} exists.
\end{lemma}
\begin{proof}
  Since $P(\tau)$ is Metzler for each $\tau\geq 0$, we have
  $\mu_{1}^+(P(\tau))=\mu_{1}(P(\tau))$, and hence $\beta=\beta^+$. In view of evident inequalities $\alpha\geq\alpha^+\geq\alpha^{++}$ and Theorem~\ref{lemma:duality}, ensuring that $\alpha\leq\beta$,
it suffices to show that $\alpha^{++}\geq\beta^+$.

  Introducing the simplex $\Delta=\setdef{x>0}{\norm{x}{1}=1}$ and its
  closure $\bar\Delta=\setdef{x\geq 0}{\norm{x}{1}=1}$, one may easily
  notice that all functions $p_i(x)=\WP{P_ix}{x}_1$ are linear on
  $\Delta$. In particular, the set
  $\mcP=\setdef{(p_0(x),\ldots,p_s(x))}{x\in\Delta}$ is convex in
  $\real^{s+1}$.  By definition of $\alpha^{++}$, the set
  $\mcY=\setdef{y=(y_0,\ldots,y_s)\in\real^{s+1}}{y_0>\alpha^{++},y_1<\rho_1,\ldots,y_s<\rho_s}$
  is disjoint with $\mcP$. Hence, a (non-strictly) separating hyperplane
  exists, i.e., there exists
  $\lambda\in\real^{s+1}\setminus\{\vectorzeros\}$ such that
  \[
  \sum_{i=0}^s\lambda_ip_i(x)\leq \lambda^{\top}y\quad\text{for all } x\in\Delta,\,y\in\mcY.
  \]
  Obviously, the latter inequality may hold only when $\lambda_0\geq 0$ and
  $\lambda_{i}\leq 0$ for $i\in\{1,\ldots,s\}$. Furthermore, passing to the
  limit as $\mcY\ni y\to(\beta^+,\rho_1,\ldots,\rho^s)$, one has
  \[
  \lambda_0(p_0(x)-\alpha^{++})+\sum_{i=1}^s\lambda_i(p_i(x)-\rho_i)\leq 0.
  \]
  Since the constraints in~\eqref{prob:primal++} are feasible, we know
  $\lambda_i(p_i(x)-\rho_i)\geq 0$ for some $x\in\Delta$ and, furthermore,
  the inequalities are strict unless $\lambda_i=0$. Therefore,
  $\lambda_0>0$.  Introducing now $\bar\tau_i=-\lambda_i/\lambda_0\leq 0$,
  for all $i\in\{1,\ldots,s\}$, one has
  \[
  \alpha^{++}\geq p_0(x)-\sum_{i=1}^s\tau_ip_i(x)+\tau^{\top}\rho=
  \WP{P(\bar\tau)x}{x}+\bar\tau^{\top}\rho, \quad \text{for all } x\in\Delta.
  \]
  Taking the supremum over all $x\in\Delta$ and applying
  Proposition~\ref{prop.M-l1} to the Metzler matrix $P(\bar\tau)$,
  one has
  $
  \alpha^{++}\geq \mu_{1}^+(P(\bar\tau))+\bar\tau^{\top}\rho\geq\beta^+.
  $
  This finishes the proof, showing also that $\bar\tau$ is a minimizer
  in~\eqref{prob:dual+} and~\eqref{prob:dual}.
\end{proof}

\subsection{Some counterexamples}

Below we demonstrate that the conditions of Lemma~\ref{lem.nogap-positive}, in fact, cannot be discarded.
In the examples below, $x=(x_1,x_2)^{\top}\in\R^2$ and $\|x\|=\|x\|_1=|x_1|+|x_2|$, we consider only one constraint $s=1$.

\textbf{Example 1.} Our first example demonstrates that the feasibility of strict inequalities~\eqref{prob:primal++}
is essential. Consider the parameters
\[
P_0=
\begin{bmatrix}
1 & 1\\
0 & 0
\end{bmatrix},
\quad P_1=\begin{bmatrix}
0 & 0\\
0 & -1
\end{bmatrix},\quad \rho_1=-1.
\]
In view of Table~\ref{table:equivalences}, $p_0(x)=(x_1+x_2,0)\sgn x=(x_1+x_2)\sgn x_1$ and
$p_1(x)=(0,-x_2)\sgn x=-|x_2|$. Hence the constraints in~\eqref{prob:primal+} are feasible
and satisfied by the unique vector $x_1=0,x_2=1$, and the supremum in~\eqref{prob:primal+} is $\alpha^+=0$.
At the same time, the constraints in~\eqref{prob:primal++} are infeasible.
Matrices $P_0$ and $(-P_1)$ are, obviously, Metzler.
Notice now that the cost function in~\eqref{prob:dual+} is constant, since
\[\
\mu_1^+(P_0-\tau P_1)+\tau\rho_1=\mu_1^+\begin{pmatrix}
1 & 1\\
0 & \tau
\end{pmatrix}-\tau=(1+\tau)-\tau=1\quad\forall\tau\geq 0,
\]
and hence $\beta^+=1>0=\alpha^+$.

\textbf{Example 2.} Our next example shows that the requirement of Metzler matrices in Lemma~\ref{lem.nogap-positive} is essential.

\textbf{A)} Consider the parameters
\[
P_0=
\begin{bmatrix}
1 & 0\\
-1 & 0
\end{bmatrix},
\quad P_1=\begin{bmatrix}
0 & 0\\
-1 & 0
\end{bmatrix},\quad \rho_1\in(-1,0).
\]
Here, matrix $(-P_1)$ is Metzler, whereas $P_0$ is not.
Then, $p_1(x)=-x_1\sgn x_2$ and the constraints in constraints in~\eqref{prob:primal++} are feasible and satisfied, e.g., by $x_1=1-\varepsilon$, $x_2=\varepsilon$ for $\varepsilon>0$ being small.
Notice, now that if $p_1(x)\leq\rho_1<0$, then
\[
p_0(x)=x_1\sgn x_1-x_1\sgn x_2\overset{x_1\sgn x_2>0}{=}|x_1|-|x_1|=0,
\]
and hence $\alpha^+=0$. At the same time, we have
\[\
\mu_1^+(P_0-\tau P_1)+\tau\rho_1=\mu_1^+\begin{pmatrix}
1 & 0\\
\tau-1 & 0
\end{pmatrix}+\tau\rho_1=
\begin{cases}
1+\tau\rho_1,&\tau\leq 1,\\
\tau(1+\rho_1),&\tau>1.
\end{cases},
\]
so the infimum in~\eqref{prob:dual+} is $\beta^+=1+\rho_1>0=\alpha^+$.

\textbf{B)} Consider now the parameters
\[
P_0=
\begin{bmatrix}
1 & 0\\
1 & 0
\end{bmatrix},
\quad P_1=\begin{bmatrix}
0 & 0\\
1 & 0
\end{bmatrix},\quad\rho_1\in (0,1/2).
\]
Now $P_0$ is Metzler, where $(-P_1)$ is not. Obviously,
$p_1(x)=x_1\sgn x_2$ and the constraints in~\eqref{prob:primal++} are feasible (being satisfied, e.g.,  by $x_1=\rho_1/2$, $x_2=1-x_1)$. Also, $p_0(x)=|x_1|+x_1\sgn x_2$ achieves its maximum at $x_1=1,x_2=0$, so the supremum in~\eqref{prob:primal+} is $\alpha^+=1$ (indeed, if $x_2>0$, then $p_0(x)=2x_1\leq 2\rho_1<1$). However,
\[\
\mu_1^+(P_0-\tau P_1)+\tau\rho_1=\mu_1^+\begin{pmatrix}
1 & 0\\
1-\tau & 0
\end{pmatrix}+\tau\rho_1=
\begin{cases}
2-\tau(1-\rho_1),&\tau\leq 1,\\
\tau(1+\rho_1),&\tau>1
\end{cases},
\]
that is, $\beta^+=1+\rho_1>\alpha^+$.

\begin{remark}
Example~2 also demonstrates the duality gap in the pair of dual problems~\eqref{prob:primal},~\eqref{prob:dual}: in both situations \textbf{A)} and \textbf{B)}, the
suprema in~\eqref{prob:primal} and~\eqref{prob:primal+} coincide $\alpha=\alpha^+$, and hence $\beta\geq\beta^+>\alpha$, because $\mu_1(A)\geq\mu_1^+(A)$ for each matrix $A$.
\end{remark}

\section{Absolute stability and absolute contractivity of Lur'e systems}\label{sec.lurie}

Consider first the \emph{Lur'e system} shown in Fig.~\ref{fig.lurie}
\begin{gather}
   \dot z(t)=Az(t)+Bw(t)\in\real^d,\quad y(t)=Cz(t)\in\real\label{eq.system}\\
   w(t)=\varphi(t,y(t))\in\real.\label{eq.nonlin}
\end{gather}

The function $\varphi$ is traditionally called ``nonlinearity'' (although one can also consider linear feedback as well). The nonlinear block need not
be static, moreover, it can have a more general form than~\eqref{eq.nonlin} and be a nonlinear operator on the whole trajectory $y(\cdot)|_0^t$.
In absolute stability theory, $\varphi$ is usually uncertain and belongs to the class of functions that obey certain \emph{quadratic constraints}~\cite{AKG-GAL-VAY:04,SVG-ALL:06,AM-AR:97}, the simplest of which are the \emph{sector} condition\footnote{In the case of static continuous nonlinearity $w(t)=\phi(y(t))$ and $-\infty<\zeta\varkappa<\infty$, the sector condition implies that $\phi(0)=0$ and $\{(y,w)| w=\phi(y))\}$ of the nonlinearity lies in the sector bounded by two lines $w=\varkappa y$ and $w=\zeta y$.}
\begin{equation}
\zeta\leq\frac{\varphi(t,y)}{y}\leq\varkappa\quad\forall y\ne 0\label{eq.sector0-}
\end{equation}
and the more sophisticated \emph{slope} condition\footnote{In the case of static differentiable nonlinearity $w(t)=\phi(y(t))$, this condition means that the minimal and maximal slope of the curve $w=\phi(y)$ are bounded, respectively, by $\zeta$ and $\varkappa$.}
\begin{equation}
\zeta\leq\frac{\varphi(t,y_1)-\varphi(t,y_2)}{y_1-y_2}\leq\varkappa\quad\forall y_1\in\real\,\forall y_2\ne y_1.\label{eq.slope0-}
\end{equation}
Here $\zeta\geq-\infty$ and $\varkappa\leq\infty$, where at least one of the latter inequalities is strict.
Note that the slope condition, generally does not entail the sector condition as exemplified, e.g., by a non-zero constant function; this implication is valid only
when $\varphi(t,0)=0$.

\subsection{Sector and slope constraints}\label{subsec.constraints}

In fact, the results on stability and contractivity of Lur'e system do not use the explicit form of the signal $w(t)$; the only information used to derive stability/contractivity criteria are the sector and slope constraints imposed on the pair of signals $w(t),y(t)$. We give a formal definition.
\begin{definition}
A pair of signals $(w(t),y(t))$ of the linear system~\eqref{eq.system} obeys the \emph{sector constraint} with sector slopes $\zeta,\varkappa$ if
\begin{equation}
\zeta y(t)^2\leq w(t)y(t)\leq \varkappa y(t)^2\;\;\forall t.\label{eq.sector0}
\end{equation}
Two input/output pairs $(w_1(t),y_1(t))$ and $(w_2(t),y_2(t))$ of~\eqref{eq.system} obeys the \emph{slope constraint} with slopes $\zeta,\varkappa$ if
\begin{equation}
\zeta \Delta y(t)^2\leq \Delta w(t)\Delta y(t)\leq \varkappa \Delta y(t)^2,\label{eq.slope0}
\end{equation}
(where, for brevity $\Delta w=w_1-w_2$ and $\Delta y=y_1-y_2$).
\end{definition}
Obviously, if the function $\varphi$ obeys the sector condition~\eqref{eq.sector0-} (respectively, the slope condition~\eqref{eq.slope0-}), then every solution (respectively, pair of solutions) of the Lur'e system~\eqref{eq.system},\eqref{eq.nonlin} obeys the sector constraint~\eqref{eq.sector0} (respectively, the slope constraint~\eqref{eq.slope0}).

Without loss of generality, we may confine ourselves to sector and slope constraints with $\zeta=0$, referring the upper slope $\varkappa\in (0,\infty]$
to as the \emph{rate} of the constraint:
\begin{align}
   0\leq wy \leq {\varkappa} y^2&\Longleftrightarrow  w(\varkappa^{-1}w-y)\leq 0,   \label{eq.sector+}\\
   0\leq \Delta w\Delta y \leq {\varkappa} \Delta y^2&\Longleftrightarrow  \Delta w(\varkappa^{-1}\Delta w-\Delta y)\leq 0.   \label{eq.slope+}
\end{align}
Otherwise, system~\eqref{eq.system} and $\zeta,\varkappa$ can be replaced by the system
\[
\dot z(t)=A'z(t)+B'v(t),\quad y(t)=Cz(t),
\]
and the new slopes $\zeta'=0$, $\varkappa'$, which are defined as
follows
\begin{enumerate}
\item in the case where $\zeta>-\infty$, we denote $v(t)=w(t)-\zeta y(t)$ and $A'=A+\zeta BC$, $B'=B$, $\varkappa'=\varkappa-\zeta$;
\item in the case where $\zeta=-\infty$, $\varkappa<\infty$, we denote $v(t)=\varkappa y(t)-w(t)$ and $A'=A-\varkappa BC$, $B'=B$, $\varkappa'=\infty$.
\end{enumerate}

\subsection{Problem setup}

Inspired by the curve norm derivative formula~\eqref{eq:cndf}, we consider
the following stability analysis problem.

\begin{problem}[Global convergence of a Lur'e system with scalar nonlinearity]
  \label{prob:stability}
  Given a norm $\norm{\cdot}{}$ with compatible weak pairing
  $\WP{\cdot}{\cdot}$, find conditions on $(A,B,C,\varkappa)$ ensuring the existence of $c>0$ satisfying the following
  \emph{Lyapunov inequality}:
  \begin{equation}
    \label{eq.lyap}
    \WP{Az+Bw}{z}\leq -c\|z\|^2\;\;\text{for all } z\in\real^d,
    w\in\real
    \enspace\text{such that $w(\varkappa^{-1}w-Cz)\leq 0$}.
  \end{equation}
\end{problem}

The inequality~\eqref{eq.lyap} allows to use the function $z\mapsto\norm{z}{}^2$ as a global Lyapunov function in order to prove the exponential convergence of each solution to $0$ or the exponential contractivity (both with rate $c$) as shown by the following simple lemma.
\begin{lemma}\label{lem:stab-contr}
Suppose that the WP $\WP{\cdot}{\cdot}$ satisfies the curve norm derivative formula~\eqref{eq:cndf} and the Lyapunov inequality~\eqref{eq.lyap}.
Then, each solution $(z,w,y)$ of~\eqref{eq.system} that obeys
the sector constraint~\eqref{eq.sector+} admits the following estimate
\begin{equation}\label{eq.abs-stab}
\|z(t)\|\leq\|z(0)\|e^{-ct}\,\forall t\geq 0.
\end{equation}
Similarly, the deviation between two solutions $(z_i,w_i,y_i)$ (where $i=1,2$) that obey the slope constraint~\eqref{eq.slope+} can be estimated as follows
\begin{equation}\label{eq.abs-contr}
\|\Delta z(t)\|\leq e^{-ct}\|\Delta z(0)\|\,\forall t\geq 0.
\end{equation}
In both situations, $t$ varies on the maximal interval where the solution(s) exist(s).
\end{lemma}
\begin{proof}
  The proof is straightforward by noticing that~\eqref{eq.lyap} entails that $D^{+} V(t)\leq -2cV(t)$, where
  $V(t)=\|z(t)\|^2$ in the case of sector constraint~\eqref{eq.sector+} and $V(t)=\|\Delta z(t)\|^2$ in the case of slope constraint~\eqref{eq.slope+}.
\end{proof}

\begin{remark}[Stability and contractivity]
In the case of Lur'e system~\eqref{eq.system},\eqref{eq.nonlin}, the point $z=\vectorzeros[d]$ is usually supposed to be an equilibrium (this is automatically implied by
~\eqref{eq.sector+} provided that $\varkappa<\infty$). In this situation,~\eqref{eq.abs-stab} is the stronger form of global exponential stability of
this equilibrium, which is called \emph{absolute stability}. The term \emph{absolute} emphasizes that the exponential stability (with a known rate) is ensured for \emph{every} nonlinear feedback $\varphi$ that obeys the sector condition and, more generally, for every input $w(t)$ that is restricted at any time by sector constraint~\eqref{eq.sector+}. The inequality~\eqref{eq.abs-contr} is guaranteed for each nonlinear feedback $\varphi$ that obeys the slope condition; in this sense it may be called the \emph{absolute contractivity}.
\end{remark}

\begin{remark}[Necessary condition for Problem~\ref{prob:stability}]
  Since $w=0$ obviously obeys the sector constraint~\eqref{eq.sector+}, the
  Lyapunov inequality~\eqref{eq.lyap} can hold only if $\WP{Az}{z}\leq
  -c\|z\|^2$, for all $z\in\real^d$.  By Lumer's equality~\eqref{eq:Lumer},
  the latter inequality is equivalent to $\mu(A)\leq-c<0$.  In sum, a
  necessary condition for Problem~\ref{prob:stability} is $\mu(A)\leq-c$.
  Note, additionally, that this condition implies $A$ is Hurwitz, since the
  log norm is an upper bound on the spectral abscissa.
\end{remark}

\begin{remark}
In fact,~\eqref{eq.lyap} can be of interest even for $c\leq 0$; in this case, we cannot guarantee convergence/contractivity of solutions, however, we still can get an estimate of the solution (respectively, the deviation of two solutions). The key result provided below (Theorem~\ref{thm.th-lp}) retains its validity for an arbitrary choice of $c\in\real$.
\end{remark}

\subsection{Transcription via weak pairings}
In this section we aim to transcribe equation~\eqref{eq.lyap} in
Problem~\ref{prob:stability}.  To this end, we introduce two matrices
\begin{equation}\label{eq.matrices1}
  P_0=
  \begin{bmatrix}
    A+cI_d & B\\
    \vectorzeros[1\times d] & 0
  \end{bmatrix}, \quad P_1=
  \begin{bmatrix}
    \vectorzeros[d\times d] & \vectorzeros[d\times 1]\\
    -C & \varkappa^{-1}
  \end{bmatrix}
\end{equation}
and the column vector $x=\begin{bmatrix} z \\ w \end{bmatrix}\in\real^{d+1}$. Obviously,
\begin{equation}
  P_0 x = \begin{bmatrix} (A+cI_d)z + Bw \\  0 \end{bmatrix}
  \quad\text{and}\quad
  P_1 x = \begin{bmatrix} \vectorzeros[d] \\  \varkappa^{-1}w -C z \end{bmatrix}.
\end{equation}

We will show that, in the case of $\ell_p$ norms,~\eqref{eq.lyap} boils down to the implication
\[
\WP{P_1x}{x}_p\leq 0\Longrightarrow\WP{P_0x}{x}_p\leq 0
\]
which can be tested via the non-polynomial S-Lemma (Theorem~\ref{lemma:duality}, statement~\ref{l:d:3}). To prove this, we need two technical lemmas.

\begin{lemma}[Sector-constraint transcription, for $\ell_p$]
  \label{lemma:constr-transcript}
  The following conditions are equivalent for each $p\in [1,\infty)$ and $x=\begin{bmatrix} z  \\ w  \end{bmatrix}\in\real^{d+1}$:
  \begin{enumerate}
  \item scalars $w$ and $y=Cz$ obey the inequality~\eqref{eq.sector+};
  \item $\WP{P_1x}{x}_p\leq 0$.
  \end{enumerate}
  In the case $p=\infty$, implication $(i)\Longrightarrow (ii)$ holds. The inverse implication $(ii)\Longrightarrow (i)$ is valid when $|w|>\|z\|_{\infty}$.
\end{lemma}
\begin{proof}
Consider first the case $p=1$.
  Noting $P_1 x = \begin{bmatrix} \vectorzeros[d] \\ -C z + \varkappa^{-1}
    w \end{bmatrix}$, we compute
  \begin{align*}
    \WP{P_1x}{x}_1 &= (\norm{z}{1}+|w|) \sign(x)^\top P_1 x
    \\
    &= (\norm{z}{1}+|w|)
    \begin{bmatrix} \sign(z)  \\ \sign(w) \end{bmatrix}^\top
    \begin{bmatrix} \vectorzeros[d] \\ \varkappa^{-1} w - C z \end{bmatrix}
    = (\norm{z}{1}+|w|)
    \sign(w) (\varkappa^{-1} w - C z).
  \end{align*}
  Clearly, $\sign(w) (-C z + \varkappa^{-1} w)\leq0$  if and only if $w(\varkappa^{-1}w-Cz)\leq 0$.

  The case $p\in(1,\infty)$ is considered in a similar way:
   \begin{align*}
    \WP{P_1 x}{x}_p &=  \|x\|_p^{2-p}(x\circ|x|^{p-2})^\top P_1x
    \\
    &= \|x\|_p^{2-p}
    \begin{bmatrix} \sign(z)\circ|z|^{2-p}  \\ \sign(w)|w|^{p-1} \end{bmatrix}^\top
    \begin{bmatrix} \vectorzeros[d] \\ -C z + \varkappa^{-1} w \end{bmatrix}\\
    &=\|x\|_p^{2-p}|w|^{p-1}\sign(w) (-C z + \varkappa^{-1} w),
  \end{align*}
  which expression is $\leq 0$ if and only if $w(\varkappa^{-1}w-Cz)\leq 0$.

  In the case $p=\infty$, one has $\WP{P_1 x}{x}_p=\max_{i\in I_{\infty}(x)}(P_1 x)_i x_i$. Obviously,
  \[
  (P_1 x)_i x_i=
  \begin{cases}
  0,\quad &i=1,\ldots,n;\\
  w(\varkappa^{-1}w-Cz),&i=n+1.
  \end{cases}
  \]
  Therefore,~\eqref{eq.sector+}  entails inequality $\WP{P_1 x}{x}_{\infty}\leq 0$, whereas the inverse statement is, generally, incorrect except for the case where $I_{\infty}(x)=\{n+1\}$, that is,
  $|w|>\|z\|_{\infty}$.
\end{proof}

To show that the first (Lyapunov) inequality in~\eqref{eq.lyap} shapes into $\WP{P_0x}{x}\leq 0$, an additional technical lemma is required. Notice that for
two vectors in $\real^{d+1}$
\begin{equation}\label{eq.aux1}
  x^1=
\begin{bmatrix}
z^1\\0
\end{bmatrix},\quad
x^2=
\begin{bmatrix}
z^2\\w^2
\end{bmatrix} \in \real^{d+1},
\end{equation}
and $p\ne 2$, generally, $\WP{x^1}{x^2}_p\ne\WP{z^1}{z^2}_p$. However, the following result holds.
\begin{lemma}
[Useful property for weak pairing]
  \label{lemma:Lyap-transcript}\begin{enumerate}
  \item For the vectors~\eqref{eq.aux1} (where $z^i\in\real^d$ and
    $w^2\in\real$) and $1\leq p<\infty$, we have
    \begin{gather}
      \WP{x^1}{x^2}_p\leq 0 \qquad\iff\qquad
      \WP{z^1}{z^2}_p\leq 0.
    \end{gather}

  \item If, additionally, $|w^2|<\|z^2\|_{\infty}$ or $w^2=\|z^2\|_{\infty}=0$, then
  $\WP{z^1}{z^2}_{\infty}=\WP{x^1}{x^2}_{\infty}$.
  \end{enumerate}
\end{lemma}
\begin{proof}
  The proof is straightforward from Table~\ref{table:equivalences}. In the
  case where $x^2=\vectorzeros$ the statement is trivial, so we always
  suppose that $\|x^2\|\ne 0$.

  If $p<\infty$, then  $\WP{z^1}{z^2}_{p}=k\WP{x^1}{x^2}_{p}$, where $k>0$ is some
  positive number. In the case $p=\infty$ (and thus
  $|w^2|<\|z^2\|_{\infty}$), one may notice that $
  \WP{x^1}{x^2}_{\infty}=\max_{i\in I_{\infty}(x^2)} x^1_ix^2_i, $ where
  $I_{\infty}(x^2)=\{j:|x^2_j|=\|x^2\|_{\infty}\}=I_{\infty}(z^2)\not\ni d+1$. Therefore,
   $\WP{x^1}{x^2}_{\infty}=\max_{i\in I_{\infty}(z^2)} z^1_iz^2_i=\WP{z^1}{z^2}_{\infty}$.
\end{proof}
\begin{corollary}[Lyapunov-inequality transcription]\label{cor:Lyap-transcript}
For all $p\in [1,\infty)$, the inequalities $\WP{Az+Bw}{z}_p\leq -c\|z\|^2_p$ and $\WP{P_0x}{x}_p\leq 0$ are equivalent. If $p=\infty$, these inequalities are also equivalent for such vectors $x$ that $x=0$ or $|w|<\|z\|_{\infty}$.
\end{corollary}
\begin{proof}
In view of~\eqref{eq:useful-eq}, the inequality $\WP{Az+Bw}{z}_p\leq -c\|z\|^2_p$ can be written as $\WP{(A+cI)z+Bw}{z}_p\leq 0$. The statement is now obvious from
Lemma~\ref{lemma:Lyap-transcript}, applied to $z^1=(A+cI)z+Bw$, $z^2=z$ and $w^2=w$.
\end{proof}

\subsection{Sufficient conditions for stability}
We are now ready to give a sufficient condition for
Problem~\ref{prob:stability} in the special situation where
$\WP{\cdot}{\cdot}=\WP{\cdot}{\cdot}_{p}$.

\begin{theorem}\label{thm.th-lp}
[Sufficient conditions for
    Problem~\ref{prob:stability} and special cases]\label{prop:lurie1}
    Consider a Lur'e system with sector constraint defined by $(A,B,C,\varkappa)$.
  Fix a norm $\ell_p$, $p\in[1,\infty]$ with log norm $\mu_p$ and weak
  pairing $\WP{\cdot}{\cdot}_{p}$. If $p=\infty$, assume also that
  $\varkappa\|C\|_{\infty}<1$.  Then the following statements hold:
  \begin{enumerate}
  \item the Lyapunov inequality~\eqref{eq.lyap} holds if
    \begin{equation}\label{prob:primal:Lure}
      \WP{P_0x}{x}_p \leq 0 \quad \text{for all }
      x\in\real^{d+1} \text{ such that } 
      \WP{P_1x}{x}_p \leq0;
    \end{equation}
furthermore,~\eqref{eq.lyap} and~\eqref{prob:primal:Lure} are equivalent when $p<\infty$;
  \item the Lyapunov inequality~\eqref{eq.lyap} holds if
    \begin{equation}\label{eq.lurie-sufficient}
      \exists\tau\geq0\;\text{such that}\;
      \mu_p(P(\tau))\leq 0\; \text{with}\;
      P(\tau)=\begin{pmatrix}
        A+cI_d & B\\
        \tau C & -\tau\varkappa^{-1}
      \end{pmatrix} 
    \end{equation}

  \item if $p=2$, then condition~\eqref{eq.lurie-sufficient} is not only
    sufficient but also necessary for the Lyapunov inequality~\eqref{eq.lyap};

  \item if $p=1$, $A$ is Metzler, and $B,C$ are non-negative, then
    condition~\eqref{eq.lurie-sufficient} is not only sufficient but also
    necessary for the Lyapunov inequality~\eqref{eq.lyap}.
\end{enumerate}
\end{theorem}

Notice that~\eqref{eq.lurie-sufficient} with $p=\infty$ in principle can
hold only when $\varkappa\|C\|_{\infty}\leq 1$. In view of this, the assumption
$\varkappa\|C\|_{\infty}<1$ is not very restrictive.
\begin{proof}

Statement (i) in the case where $p\in[1,\infty)$ is straightforward from Lemma~\ref{lemma:constr-transcript} and Corollary~\ref{cor:Lyap-transcript}.
To prove it for $p=\infty$, one has to use the assumption $\varkappa\|C\|_{\infty}<1$ entailing that $|w|\leq\varkappa|Cz|\leq \varkappa\|C\|_{\infty}\|z\|_{\infty}$, that is, $|w|<\|z\|_{\infty}$ unless $w=\|z\|=0$. Corollary~\ref{cor:Lyap-transcript} entails that the first inequality in~\eqref{eq.lyap} is
  also equivalent to $\WP{P_0x}{x}_{\infty}\leq 0$. The second inequality in~\eqref{eq.lyap} also implies the relation $\WP{P_1x}{x}_{\infty}\leq 0$
  (Lemma~\ref{lemma:constr-transcript}). Hence,~\eqref{prob:primal:Lure} is sufficient (yet not necessary) for~\eqref{eq.lyap} also for $p=\infty$.

Statement $(ii)$ is now straightforward from Theorem~\ref{lemma:duality}.

Statements $(iii)$ and $(iv)$ follow, respectively, from Lemmas~\ref{lem.yakub} and~\ref{lem.nogap-positive} (one can easily notice that $P_0$ and $(-P_1)$ are Metzler matrices).
\end{proof}

Theorem~\ref{prop:lurie1} can be generalized to the weighted norm
case. Introducing a new variable $\tilde z=Rz$, one easily shows
that~\eqref{eq.lyap} holds for $\WP{\cdot}{\cdot}=\WP{\cdot}{\cdot}_{p,R}$,
where $R\in\real^{d\times d}$ is an invertible matrix, if and only if it
holds for $\WP{\cdot}{\cdot}=\WP{\cdot}{\cdot}_{p}$ and matrices $\tilde
A=RAR^{-1}$, $\tilde B=RB$, $\tilde C=CR^{-1}$. Introducing the
block-diagonal matrix $\tilde R={\rm diag}(R,1)$, one obtains the following
corollary.
\begin{corollary}\label{cor.lurie2}
Suppose that $\WP{\cdot}{\cdot}=\WP{\cdot}{\cdot}_{p,R}$, where $1\leq p\leq\infty$ and, if $p=\infty$ one has $\varkappa\|CR^{-1}\|_{\infty}<1$. Then~\eqref{eq.lyap} holds if $\tau\geq 0$ exists such that
\begin{equation}\label{eq.lurie-sufficient-R}
\mu_{p,\tilde R}(P(\tau))=\mu_p\begin{pmatrix}
RAR^{-1}+cI_d & RB\\
\tau CR^{-1} & -\tau\varkappa^{-1}
\end{pmatrix}\leq 0.
\end{equation}
If $p=2$,~\eqref{eq.lurie-sufficient-R} is also necessary for~\eqref{eq.lyap}; the same holds when $p=1$, $RAR^{-1}$ is a Metzler matrix and matrices $RB,CR^{-1}$ are nonnegative.
\end{corollary}

\begin{remark}[Convergence and contractivity in a single criterion]
In view of Lemma~\ref{lem:stab-contr}, Theorem~\ref{prop:lurie1} and Corollary~\ref{cor.lurie2} give simultaneously a criterion of global exponential convergence for solutions that obey the sector constraint~\eqref{eq.sector+}
and a criterion of exponential contraction for solutions obeying the slope constraint~\eqref{eq.slope+}.
\end{remark}

We would like to notice that, for fixed $p\in [1,\infty]$ and fixed invertible matrix $R$, the feasibility of~\eqref{eq.lurie-sufficient-R} reduces to a \emph{convex program}
\begin{equation}\label{eq.convex}
\begin{gathered}
\inf_{c,\tau}\quad\mu_{p,R} 
      \begin{pmatrix}
        A+cI_d & B\\
        \tau C & -\tau\varkappa^{-1}
      \end{pmatrix},\\
      \text{subject to: $c>0,\tau\geq 0$.}
    \end{gathered}
\end{equation}
The condition~\eqref{eq.lurie-sufficient-R} holds if and only if the answer is non-positive.
Unfortunately, the logarithmic norm can be computed
efficiently only for $p=1,2,\infty$, where its computation in the case of general $p$ is a difficult operation.\footnote{Note that
  computation of the matrix operator norm $\|\cdot\|_p$ with $p\in
  (1,\infty)\setminus\{2\}$ (even up to a constant factor) proves to be an
  NP-hard problem~\cite{AB-AV:11}.} Regarding $R$ as a parameter, the problem becomes in general
\emph{non-convex} as the function $\mu_{p,R}$ is non-convex in
$R$. However, when $p=2$ or when $p\in\{1,\infty\}$ and $R$ is diagonal,
one can show that the problem is quasiconvex in $R$ (after a
transcription), at fixed $c$ and $\tau$.  Therefore, the combined problem
of computing $R$ and searching for the contraction factor $c$ can be solved
via a bisection algorithm. A discussion of the $p=2$ case is given in the
classic text on LMIs in control~\cite[Section~5.1]{SB-LEG-EF-VB:94} and a
discussion of the $p\in\{1,\infty\}$ for the case of diagonal weights $R$
is given in~\cite[Section~2.7]{FB:22-CTDS}. We discuss both cases in some
detail below.

\subsection{The $\ell_2$ norm case: comparison with known results}

In the case $p=2$, Theorem~\ref{thm.th-lp} and Corollary~\ref{cor.lurie2} lead to several known results.

\subsubsection{Analytic form of the condition~\eqref{eq.lurie-sufficient}}

Recall that for an arbitrary square matrix $M$ one has $\mu_2(M)=\lambda_{\max}(M^s)$, where $M^s=(M+M^{\top})/2$ is the symmetric part of $M$. In other words,
$\mu_2(M)\leq 0$ (respectively, $<0$) if and only if $M^s\preceq 0 $ (respectively, $M^s\prec 0$).
This allows us to simplify the condition~\eqref{eq.lurie-sufficient}. It can be easily shown that~\eqref{eq.lurie-sufficient} (with $p=2$) cannot hold with $\tau=0$, except for the trivial situation where $B=0$. For this reason, we are interested only in $\tau>0$.
\begin{corollary}\label{cor.p2-simplification}
Inequality~\eqref{eq.lyap} with $p=2$, $\varkappa<\infty$, $\tau>0$ holds if and only if
\begin{equation}\label{eq.p2-simplification}
  \exists \tau>0: A^s+cI+\frac{\varkappa}{4\tau}(B+\tau C^{\top})(B^{\top}+\tau C)\preceq 0.
\end{equation}
\end{corollary}
\begin{proof}
The left-hand side of~\eqref{eq.p2-simplification} is nothing else than the Schur complement of the bottom-right diagonal entry in the matrix
\[P(\tau)^s=
\begin{bmatrix}
A^s+cI & \frac{1}{2}(B+\tau C^{\top})\\
* & -\tau\varkappa^{-1}
\end{bmatrix}.\]
Since $\tau\varkappa^{-1}<0$, the Haynsworth~\cite{EVH:68} theorem implies that $P(\tau)^s\preceq 0$ (i.e.,~\eqref{eq.lurie-sufficient} holds with $p=2$) if and only if this Schur complement is negative semidefinite.
\end{proof}

In stability and contraction problems, one is primarily interested in the \emph{existence} of such a convergence/contractivity rate $c>0$ that~\eqref{eq.p2-simplification} holds, that is, in the validity of the strict inequality as follows
\begin{equation}\label{eq.p2-simplification+}
  \exists \tau>0: A^s+\frac{\varkappa}{4\tau}(B+\tau C^{\top})(B^{\top}+\tau C)\prec 0.
\end{equation}
Notice that~\eqref{eq.p2-simplification} can be rewritten as
\begin{equation}\label{eq.aux}
\begin{gathered}
\exists \tau>0:\quad\mathcal{A}(\varkappa)+\frac{\varkappa}{4\tau}(B-\tau C^{\top})(B^{\top}-\tau C)\prec 0,\\
\mathcal{A}(\varkappa):=A^s+\varkappa(BC)^s.
\end{gathered}
\end{equation}
In particular,~\eqref{eq.p2-simplification} entails that $\mathcal{A}(\varkappa)\prec 0 $. This condition, however, is only necessary yet not sufficient for~\eqref{cor.p2-simplification}. A necessary and sufficient condition is provided by the following lemma, relating the result of our Theorem~\ref{thm.th-lp}
and~\cite[Theorems~1 and~3]{RD-SD:18}.


\begin{lemma}\label{lem.p2}
The following statements are equivalent:
\begin{enumerate}
  \item\label{l:p2:1} condition~\eqref{eq.p2-simplification+} (equivalently,~\eqref{eq.aux}) holds;
  \item\label{l:p2:2} the family of inequalities~\eqref{eq.drummond} is valid
  \begin{equation}\label{eq.drummond}
  \mathcal{A}(\kappa)\prec 0 \quad\forall\kappa\in[0,\varkappa].
  \end{equation}
  \item\label{l:p2:3} $\mathcal{A}(\varkappa)\prec 0 $ and the matrix $S=(-\mathcal{A}(\varkappa))^{-1/2}\succ 0 $ satisfies the condition
  \begin{equation}\label{eq.drummond1}
  \|\tilde B\|\|\tilde C\|-\tilde C\tilde B\leq 2\varkappa^{-1},\quad
  \tilde B:=SB,\,\tilde C:=CS.
  \end{equation}
\end{enumerate}
\end{lemma}
\begin{proof}
To show that \ref{l:p2:1} implies \ref{l:p2:2}, notice first that the validity of~\eqref{eq.p2-simplification} with rate $\varkappa$ entails its validity with
every smaller rate $\kappa\in[0,\varkappa]$. Using the decomposition~\eqref{eq.aux} (with $\kappa$ instead of $\varkappa$), one shows that $\mathcal{A}(\kappa)\prec 0 $.

To prove the implication \ref{l:p2:2}$\Longrightarrow$\ref{l:p2:3}, notice that inequalities~\eqref{eq.drummond} can be written as follows:
$S^{-2}+(\varkappa-\kappa) (BC)^s\succ 0 \,\forall \kappa\in [0,\varkappa]$ or, equivalently, $I+k(\tilde B\tilde C)^s\succ 0 \,\forall k\in [0,\varkappa]$. The latter condition can be formulated as follows: the minimal eigenvalue $\lambda_{\min}\left((\tilde B\tilde C)^s\right)\geq -\varkappa^{-1}$.
The eigenvalues of $(\tilde B\tilde C)^s$ are easy to compute: two of them equal
$
(\tilde C\tilde B\pm \|\tilde B\|\|\tilde C\|)/2
$
and the $(n-2)$ are zero (if $B$ and $C$ are parallel, then matrix has only one non-zero eigenvalue). Hence,~\eqref{eq.drummond} can be written
as follows: $\tilde C\tilde B-\|\tilde B\|\|\tilde C\|\geq -2\varkappa^{-1}$, which is equivalent to~\eqref{eq.drummond1}.

To prove the final implication \ref{l:p2:3}$\Longrightarrow$\ref{l:p2:1}, recall that $\mathcal{A}(\varkappa)=-S^2$. Hence~\eqref{eq.aux} can be equivalently written as follows: $\tau>0$ exists such that
\[
-I+\frac{\varkappa}{4\tau}(\tilde B-\tau\tilde C^{\top})(\tilde B^{\top}-\tau\tilde C)\prec 0,
\]
which holds if and only if $\|\tilde B-\tau\tilde C^{\top}\|^2<4\tau\varkappa^{-1}$. It can be easily checked that~\eqref{eq.drummond1} entails the latter inequality
with $\tau=\|\tilde B\|/\|\tilde C\|$. This finishes the proof.
\end{proof}

The inequalities~\eqref{eq.drummond} were introduced in~\cite{RD-SD:18} as
a condition for the absolute stability (in the case of nonlinearities with
sector condition) and the absolute contractivity (in the case of slope
condition). Lemma~\ref{lem.p2} shows that this condition is a special case
of the standard S-Lemma for the $\ell_2$-norm. Statement~\ref{l:p2:3} gives
an efficient way to test the condition~\eqref{eq.drummond}. Furthermore, the classical Yakubovich S-Lemma
(Lemma~\ref{lem.yakub}) shows that, in fact,~\eqref{eq.lurie-sufficient} is
\emph{necessary} for the Lyapunov property~\eqref{eq.lyap}, and
hence~\eqref{eq.drummond} are in fact minimal conditions under which absolute stability~\eqref{eq.abs-stab} (with sector constraint) and absolute contractivity~\eqref{eq.abs-contr} (with slope constraint) in $\ell_2$ norm can be established.

\subsubsection{The LMI stability criterion and KYP lemma}

In this subsection, we consider the standard reformulation of Corollary~\ref{cor.lurie2} in the case $p=2$.
Without loss of generality we may assume that $R$ in Corollary~\ref{cor.lurie2} is symmetric positive definite, and thus it can be written as $R=H^{1/2}$, where
$H=H^{\top}\succ 0$. Indeed, $\WP{x}{y}_{2,R}=\WP{R^{\top}Rx}{y}_2=\WP{x}{y}_{2,H^{1/2}}$, where
$H=R^{\top}R\succ 0$ and where $R$ is invertible). The inequality~\eqref{eq.lurie-sufficient-R} with $p=2$ is therefore equivalent to
\[
\begin{bmatrix}
RAR^{-1}+cI_d & RB\\
\tau CR^{-1} & -\tau\varkappa^{-1}
\end{bmatrix}+\begin{bmatrix}
RAR^{-1}+cI_d & RB\\
\tau CR^{-1} & -\tau\varkappa^{-1}
\end{bmatrix}^{\top}\leq 0,
\]
which, recalling that $R=H^{1/2}>0$, further simplifies to the condition~\citep{LDA-MC:13}
\begin{equation}\label{eq.kyp-lmi}
\begin{bmatrix}
HA+A^{\top}H+2cH & HB+\tau C\\
* & -\tau\varkappa^{-1}
\end{bmatrix}\leq 0.
\end{equation}
This condition, obviously, cannot hold for $\tau=0$ (except for the
degenerate case where $B=0$). If we are interested in the existence of some
$H\succ 0,c>0$, then we may assume, without loss of generality that $\tau=1$. If
the triple $(A,B,C)$ is controllable and observable, the
Kalman-Yakubovich-Popov lemma~\cite{AKG-GAL-VAY:04} implies that the
existence of $H=H^{\top}\succ 0$, $c>0$ obeying~\eqref{eq.kyp-lmi} is equivalent\footnote{The KYP lemma (relying only on the controllability and observability) guarantees, in fact, that the strict inequality in~\eqref{eq.kyp-lmi} holds for appropriate $c>0,H=H^{\top}$; formally, $H$ is not guaranteed to be positive definite. However, the strict inequality in~\eqref{eq.kyp-lmi} entails that $H\succ 0$ when $A$ is Hurwitz.} to the frequency-domain condition
\begin{equation}\label{eq.kyp}
\max_{\omega\in\R}\realpart\big(C(\imath\omega I-A)^{-1}B\big) < \varkappa^{-1}.
\end{equation}
This is a special case of the so-called ``circle criterion'' for absolute stability~\cite{AKG-GAL-VAY:04}, in which the circle degenerates to a half-plane on $\mathbb{C}$.

\subsection{A sufficient condition for the diagonally weighted $\ell_p$ norm}

This subsection offers a simple condition
ensuring~\eqref{eq.lurie-sufficient-R} for an arbitrary $p\in[1,\infty]$
and some \emph{diagonal} weight matrix $R$. This condition is inspired by
properties of Metzler matrices and their logarithmic norms. While this Metzler-based approach is
  potentially conservative for general triples of matrices $(A,B,C)$, we refer the reader to the discussion in the
  introduction about the advantages of non-Euclidean norms.

Recall that $\metzler{\cdot}$ stands for the Metzler majorant of a matrix
and $|\cdot|$ stands for the entry-wise absolute value; we also use
$\alpha(A)=\max\realpart{\lambda_j(A)}$ to denote the spectral abscissa of
matrix $A$. The following lemma is not widely known in the control literature, but
has been established in a sequence of paper on linear algebra,
see~\citep[Theorem~2]{JS-CW:62},~\citep{JA:96},
and~\citep[Lemma~3]{OP-MV:06}. We provide its proof for readers' convenience.
\begin{lemma}\label{lem.metzler}
For each $\varkappa<\infty$, consider the Metzler matrix $\mathfrak{A}(\varkappa)=\metzler{A}+\varkappa |B|\,|C|$.
The following two statements hold:
\begin{enumerate}
  \item\label{l:m:1} If $\alpha(\mathfrak{A}(\varkappa))<-c$ (that is, $\metzler{A}+\varkappa |B|\,|C|+cI$ is a Hurwitz matrix), then
  for \emph{every} $p\in[1,\infty],\tau>0$ there exists a diagonal matrix $R=R(p,\tau)\succ 0 $ such that~\eqref{eq.lurie-sufficient-R} holds and, in the case of $p=\infty$, $\varkappa\|CR^{-1}\|_{\infty}<1$.
  \item\label{l:m:2}On the other hand, if~\eqref{eq.lurie-sufficient-R} holds with some diagonal $R\succ 0$ and $p\in\{1,\infty\}$, then $\alpha(\mathfrak{A}(\varkappa))\leq -c'$, i.e., the condition from~\ref{l:m:1} holds for each $c'\in (0,c)$.
\end{enumerate}
\end{lemma}
\begin{proof}[Proof of Lemma~\ref{lem.metzler}]
To prove~\ref{l:m:1}, we first apply~\cite[Lemma~2]{YE-DP-DA:17} (Schur complements for Metzler Hurwitz matrices) to
the Metzler majorant of the matrix $P(\tau)$
\[
\metzler{P(\tau)}=
\begin{bmatrix}
\metzler{A}+cI & |B|\\
\tau|C| & -\tau\varkappa^{-1}
\end{bmatrix},\quad \tau>0
\]
showing that $\metzler{P(\tau)}$ is Hurwitz if (and in fact, only if) the Metzler matrix $\mathfrak{A}(\varkappa)+cI=\metzler{A}+\varkappa |B|(\tau\varkappa^{-1})^{-1}\tau|C|+cI$ is Hurwitz, that is, $\alpha(\mathfrak{A}(\varkappa))<-c$.

Hence,~\cite[Lemma~3]{OP-MV:06} (the lemma on optimally diagonally weighted log norms for Metzler matrices) ensures the existence of a diagonal matrix $\tilde R\succ 0$ such that $\mu_{p,\tilde R}(\metzler{P(\tau)})<0$. Rescaling $\tilde R$, one mays assume without loss of generality that $\tilde R={\rm diag}(R,1)$, where $R\succ 0$.
Noticing that the norm $\|\cdot\|_{p,\tilde R}$ is \emph{monotone}, that is, for each two vectors $x,y\in\real^{d+1}$ such that $|x_i|\leq|y_i|\,\forall i$, one has $\|x\|_{p,\tilde R}\leq\|y\|_{p,\tilde R}$ and applying~\cite[Theorem~2.23]{FB:22-CTDS}, one proves that~\eqref{eq.lurie-sufficient-R}:
\[
\mu_p\begin{pmatrix}
RAR^{-1}+cI_d & RB\\
\tau CR^{-1} & -\tau\varkappa^{-1}
\end{pmatrix}
=\mu_{p,\tilde R}(P(\tau))\overset{\text{monotonicity}}{\leq} \mu_{p,\tilde R}(\metzler{P(\tau)})<0.
\]
Finally, if $p=\infty$, one can also notice that
\[
-\tau\varkappa^{-1}+\tau\|CR^{-1}\|_{\infty}\leq -\tau\varkappa^{-1}+\tau\|CR^{-1}\|_1\overset{(*)}{\leq} \mu_{\infty}\begin{pmatrix}
RAR^{-1}+cI_d & RB\\
\tau CR^{-1} & -\tau\varkappa^{-1}
\end{pmatrix}<0.
\]
Here the inequality $(*)$ holds due to the representation of $\mu_{\infty}$ (Table~\ref{table:equivalences}).

To prove~\ref{l:m:2}, notice that $\alpha(\metzler{P(\tau)})\leq\mu_{p,\tilde R}(\metzler{P(\tau)})=\mu_{p,\tilde R}(P(\tau))\leq 0$ when $p\in\{1,\infty\}$ and $\tilde R\succ 0$ is a diagonal matrix.
Hence, for each $\varepsilon>0$ the matrix
\[
\metzler{P(\tau)}-(\tau\varkappa^{-1}\varepsilon) I_{d+1}
=
\begin{bmatrix}
\metzler{A}+(c-\tau\varkappa^{-1}\varepsilon) I_{d} & |B|\\
\tau |C| & -\tau\varkappa^{-1}(1+\varepsilon)
\end{bmatrix}
\]
is Hurwitz. Invoking~\cite[Lemma~2]{YE-DP-DA:17}, one proves that
$A+(c-\tau\varkappa^{-1}\varepsilon) I_{d}+\varkappa(1+\varepsilon)^{-1}|B|\,|C|=\mathfrak{A}(\varkappa)+cI_d+O(\varepsilon)$ is a Hurwitz matrix for an arbitrary small $\varepsilon>0$. Statement~\ref{l:m:2} is now proved by taking the limit as $\varepsilon\to 0$.
\end{proof}
\begin{remark}\label{rem.positive-necess-suff}
  Recall that, in view of Corollary~\ref{cor.lurie2}, for the case where
  $A$ is Metzler and $B,C$ are nonnegative, the Lyapunov
  condition~\eqref{eq.lyap} holds for the norm $\|\cdot\|_{1,R}$, where
  $R\succ 0$ is diagonal matrix, \emph{if and only
    if}~\eqref{eq.lurie-sufficient-R} is valid. The condition
  $\alpha(\mathfrak{A}(\varkappa))<0$ is hence the \emph{necessary and
    sufficient} condition for the absolute stability~\eqref{eq.abs-stab}
  (with sector constraint) and absolute contractivity~\eqref{eq.abs-contr}
  (with slope constraint) with respect to diagonally weighted $\ell_1$-norms.
  \end{remark}

\subsection{A discussion on Aizerman and Kalman conjectures}

It is interesting to compare the results of Lemma~\ref{lem.p2} and
Lemma~\ref{lem.metzler} with the two famous conjectures that were
formulated at the dawn of nonlinear control theory. Aizerman~\cite{MAA:49}
conjectured that for the global stability of the equilibrium $x=0$ in the
Lur'e system~\eqref{eq.system},\eqref{eq.nonlin} with an arbitrary
continuous nonlinearity in sector~\eqref{eq.sector0} it suffices to prove
stability with all linear feedback functions $\varphi(y)=ky$, where
$k\in[\zeta,\varkappa]$, that is,
\begin{equation}\label{eq.aizerman}
  \alpha(A+kBC)<0\quad\forall k\in[\zeta,\varkappa].
\end{equation}
(all matrices $A+kBC$ are Hurwitz). Notice that~\eqref{eq.aizerman} can be efficiently tested, e.g., via the Nyquist criterion. Later, Kalman~\cite{REK:57} conjectured that~\eqref{eq.aizerman} guarantees global stability of~\eqref{eq.system},\eqref{eq.nonlin} with every differentiable nonlinearity such that $\zeta\leq\varphi'(y)\leq\varkappa$.

In general, neither Aizerman's nor even Kalman's conjecture proves to be valid when the dimension of the state vector is $d\geq 3$; the reader is referred to~\cite{GAL-NVK:13,RD-SD:18} for the survey of main historical milestones and recent achievements in this area. However, these conjectures may be valid for special triples of matrices $(A,B,C)$.

The criterion from Lemma~\ref{lem.p2} (partly available in~\cite{RD-SD:18}) shows that the Aizerman conjecture is valid in the situation where $A$ is symmetric and $B,C^{\top}$ are parallel (in this situation, $BC$ is also a symmetric matrix). Note that formally we have considered only the sector with $\zeta=0$, however, the transformation introduced in Subsection~\ref{subsec.constraints} allows to discard this assumption. Lemma~\ref{lem.metzler} shows that the Aizerman conjecture is valid for the case where $A$ is Metzler and $B,C$ are nonnegative and the sector's lower slope is $\zeta=0$. Notice that this fact is typically proved by using the diagonally weighted $\ell_2$ norm as a Lyapunov function~\cite{MYC:10}; Lemma~\ref{lem.metzler} shows that, in fact, one can use diagonally weighted $\ell_p$ norm with an arbitrary choice of $p\in[1,\infty]$. In both situation, the stronger version of Kalman's conjecture proves to be valid: if the nonlinearity is slope-restricted in the sense that
$0\leq\varphi'(y)\leq\varkappa$, then the Lur'e system~\eqref{eq.system},\eqref{eq.nonlin} enjoys the exponential global contractivity property.

\subsection{Generalization to MIMO nonlinear blocks}

To keep the presentation in this paper simple, we have confined ourselves
to the classical Lur'e system with a scalar nonlinear block.  However, the
general form of Theorem~\ref{lemma:duality} with $s>1$ constraints allows us
to extend the results of these paper to many kinds of multidimensional
nonlinearities, e.g., \emph{diagonal} nonlinearities $w(t)=\Phi(y(t))$
where $\dim y=\dim w=p$ and $\Phi(y)$ is a diagonal matrix whose $i$th
diagonal entry $\Phi(y)_{ii}=\varphi_i(y_i)$ depends only on
$y_i$. Assuming that all scalar functions $\varphi_i(y_i)$ obey the sector
or slope constraint~\eqref{eq.sector+} or~\eqref{eq.slope+} (where $w,y$
have to be replaced by $w_i,y_i$), Theorem~\ref{thm.th-lp} retains its
validity with the only difference that in~\eqref{eq.lurie-sufficient}
$\tau$ is not a scalar but a diagonal matrix $\tau={\rm
  diag}(\tau_1,\ldots,\tau_p)$; the corollaries of Theorem~\ref{thm.th-lp}
thus also can be generalized to the case of Lur'e-type systems with MIMO
nonlinearities. Such systems arise in a broad range of applications, e.g.,
in dynamical models of neural circuits~\cite{EK-AB:00,AD-AVP-FB:21k}.

\section{Conclusions and future works}\label{sec.concl}

One of keystones of modern nonlinear control theory, the S-Lemma serves a
convenient tool for Lyapunov stability and contractivity analysis of
nonlinear systems. This lemma enables one to transform quadratic
constraints on nonlinearities (e.g., sector or slope constraints) into the
Lyapunov condition $\dot V\leq -cV$, where $V$ is a positive definite
quadratic form of the state vector (or, in contraction analysis, state
increment) and $c\geq 0$ is the convergence or contraction rate. Quadratic
Lyapunov functions allow us to estimate the convergence and contraction
rate in some \emph{Euclidean} norm (that is, a norm induced by an inner
product). To obtain such estimates in \emph{non-Euclidean} norms, e.g., the
$\ell_p$ norms, we generalize the classical S-Lemma to non-quadratic
functions that are induced by the \emph{weak pairing} associated with the
norm. Using this generalized S-Lemma, we derive novel criteria for absolute
stability and contractivity that are based on $\ell_p$ norms (possibly,
weighted) and give alternative proofs of some recent results, in
particular, symmetrization-based stability criteria from~\cite{RD-SD:18}
and the Aizerman conjecture for positive Lur'e systems~\cite{MYC:10}.


A topic of ongoing research is to tighten the criteria of absolute stability
and contractivity by accounting additional properties of nonlinearities,
for instance, their boundedness as in~\cite{MF-MM-GJP:20}; {we believe that the relevant extension will be helpful in $\ell_1/\ell_{\infty}$ robustness, reachability and safety analysis of neural networks.}
Another direction of ongoing research is to obtain efficient numerical algorithms for validation of
the conditions~\eqref{eq.lurie-sufficient} for $p\ne 1,2,\infty$. Although the exact computation of the log norm is troublesome,
some upper estimates can potentially be employed.

\bibliographystyle{siamplain}
\bibliography{alias,Main,FB,New}
\appendix\section{Proof of Proposition~\ref{prop.M-l1}}

  Without loss of generality, we may assume that $M$ is non-negative
  (otherwise, replace $M$ by the matrix $M+\alpha I$ with $\alpha$ chosen
  large enough; obviously, this operation increases all three parts
  of~\eqref{eq.conic-metzler} by $\alpha$).  The first equality
  in~\eqref{eq.conic-metzler} is straightforward from Table~\ref{table:equivalences}: $|M_{ij}|=M_{ij}^+\geq 0\,\forall i\ne j$.

  Considering the open simplex $\Delta=\setdef{x>0}{\norm{x}{1}=1}$ and its
  closure $\bar\Delta=\setdef{x\geq 0}{\norm{x}{1}=1}$, we know (Table~\ref{table:equivalences}) that
  $\mu_{1}^+(M)=\sup_{x\in\bar\Delta}\WP{Mx}{x}\geq
  \sup_{x\in\Delta}\WP{Mx}{x}$. To prove the inverse inequality, notice
  that for $M$ being non-negative, the function $\WP{Mx}{x}$ is
  \emph{concave} on $\bar\Delta$. Indeed, denoting $I(x)=\{i:x_i>0\}$, one
  has
  \[
  \WP{Mx}{x}_1=\sum_{i\in I(x)}(Mx)_i.
  \]
  Given two vectors $x^0,x^1\in\Delta$, $\theta\in (0,1)$ and denoting
  $x^{\theta}=\theta x^1+(1-\theta)x^0$, one has $I(x^{\theta})=I(x^0)\cup I(x^1)$, $Mx^{\theta}\geq\theta (Mx^1)\geq
  0$ and $Mx^{\theta}\geq(1-\theta) Mx^0\geq 0$.  Thus
  \[
  f(\theta):=
  \sum_{i=1}^n(Mx^{\theta})_i\geq (1-\theta)\sum_{i\in I(x^0)}(Mx^0)_i+\theta\sum_{i\in I(x^1)}(Mx^1)_i=(1-\theta)f(0)+\theta f(1).
  \]
  In particular, $\liminf_{\theta\to 0^+}f(\theta)\geq f(0)$. Choosing
  $x^0$ and $x^1$ in such a way that $x^{\theta}>0$ for each $\theta\in
  (0,1)$ (that is, $I(x^0)\cup I(x^1)=\{1,\ldots,n\}$), one proves that
  \[
  \WP{Mx^0}{x^0}_1\leq\liminf_{\theta\to 0^+}\WP{Mx^{\theta}}{x^{\theta}}\leq \sup_{x\in\Delta}\WP{Mx}{x}_1.
  \]
  Since, $x^0\in\bar\Delta$ can be arbitrary, we have
  $\mu_{1}^+(M)=\sup_{x\in\bar\Delta}\WP{Mx}{x}\leq
  \sup_{x\in\Delta}\WP{Mx}{x}$, which proves the second equality
  in~\eqref{eq.conic-metzler}.\qed

\end{document}